\documentclass[10pt,openany]{book}

\usepackage{natbib}

\usepackage{array}

\usepackage{color}

\usepackage{amsmath}
\usepackage{amssymb}
\usepackage{amsfonts}
\usepackage{amsthm}
\usepackage{rotating}
\usepackage{mathtools}
\usepackage{dsfont}
\usepackage{array}
\usepackage{multirow}
\usepackage{multicol}

\usepackage{enumerate}

\usepackage{graphicx}

\usepackage{tabularx}
\usepackage{longtable}
\usepackage{lscape}
\usepackage{float}

 \usepackage[ruled,lined,commentsnumbered]{algorithm2e}
 
\newtheorem{theor}{Theorem}
\newtheorem{lem}{Lemma}
\newtheorem{defi}{Definition}
\newtheorem{rem}{Remark}

\newcommand{\id}{
\textrm{I}
}

\newcommand{\XX}{\mathbf{X}}

\newcommand{\XXi}{X_i}
\newcommand{\XuXu}{\mathbf{X}_u}
\newcommand{\XvXv}{\mathbf{X}_v}
\newcommand{\xixi}{{x_i}}
\newcommand{\xx}{\mathbf{x}}
\newcommand{\xuxu}{\mathbf{x}_u}
\newcommand{\xvxv}{\mathbf{x}_v}
\newcommand{\xjxj}{{x_j}}

\newcommand{\LRi}{L^2_{\mathbb{R},i}}

\newcommand{\PX}{P_\mathbf{X}}
\newcommand{\PXi}{P_{\mathbf{X}_i}}
\newcommand{\pX}{p_\mathbf{X}}
\newcommand{\pXu}{p_{\XuXu}}

\newcommand{\cov}{\mathrm{Cov}}
\newcommand{\E}{\mathbb{E}}

\newcommand{\abs}[1]{\left\lvert #1 \right\rvert}
\newcommand{\psemp}[2]{\langle #1,#2 \rangle_{\tiny n}}
\newcommand{\psempn}[2]{\langle #1,#2 \rangle_{  n_1}}
\newcommand{\psempnn}[2]{\langle #1,#2 \rangle_{n_2}}
\newcommand{\psth}[2]{\langle #1,#2 \rangle}

\newcommand{\norm}[1]{\left\| #1 \right\|}
\newcommand{\normemp}[1]{\left\| #1 \right\|_{\tiny n}}
\newcommand{\eucl}[1]{\left\| #1 \right\|_{\text{\tiny 2}}}

\newcommand{\normmat}[1]{
\left\lvert\!\left\lvert\!\left\lvert #1 \right\rvert\!\right\rvert\!\right\rvert_{\text{\tiny 2}}
}

\newcommand{\disc}[2]{[#1 : #2]}
\newcommand{\gener}[1]{\mbox{Span}\left\{ #1 \right\}}
\newcommand{\bolds}[1]{\boldsymbol{#1}}
\newcommand{\lu}{
\boldsymbol{l_u}
}

\newcommand{\lij}{
\boldsymbol{l_{ij}}
}

\newcommand{\incchap}{\hat{\bolds\lambda}_{n_1}}
\newcommand{\inc}{\bolds\lambda}

\newcommand{\vf}[2]{\hat\phi_{#1, n_1}^{#2}}
\newcommand{\vfu}{\hat\phi_{\bolds{l_u}, n_1}^{u}}
\newcommand{\vfv}{\hat\phi_{\bolds{l_v}, n_1}^{v}}
\newcommand{\vfs}{\hat\phi_{\bolds{l_{u_k}}, n_1}^{u_k}}
\newcommand{\vfij}{\hat\phi_{\bolds{\bolds{l_{ij}}}, n_1}^{ij}}
\newcommand{\vfi}{\hat\phi_{l_i,  n_1}^{i}}

\newcommand{\vfsbis}{\hat \phi_{\bolds{l_{u_j}},n_1}^{u_j}}

\newcommand{\pf}[2]{\phi_{#1}^{#2}}
\newcommand{\pfu}{\phi_{\bolds{l_u}}^{u}}
\newcommand{\pfv}{\phi_{\bolds{l_v}}^{v}}
\newcommand{\pfs}{\phi_{\bolds{l_{u_k}}}^{u_k}}
\newcommand{\pfij}{\phi_{\bolds{\bolds{l_{ij}}}}^{ij}}

\newcommand{\pfi}{\phi_{l_i}^{i}}
\newcommand{\pfj}{\phi_{l_j}^{j}}

\newcommand{\cvpro}{ \overset{\mathbb{P}}{\longrightarrow}}

\newcommand{\cvproinf}{ \xrightarrow[n \rightarrow + \infty]{\mathbb{P}}}

\newcommand{\gop}{\mathcal O_P}
\newcommand{\gopxi}{\mathcal O_P(n^{-\xi/2})}
\newcommand{\gopxiv}{\mathcal O_P(n^{\vartheta-\xi/2})}
\newcommand{\gophyp}{\underset{n\rightarrow + \infty}{\mathcal O}(\exp (Cn^{1-\xi}))}

\newcommand{\betasum}{\sum_{\substack{u\in S^*\\ |u|\leq d}} \sum_{\lu}\abs{\beta_{\lu}^{u,0}}}

\newcommand{\betasumm}{\| \bolds{\beta^0}\|_{L^1}}

 \newcommand{\argmin}{\operatornamewithlimits{Argmin}}

\newcommand\beq{\begin{equation}}
\newcommand\eeq{\end{equation}}

\newcommand\beqe{\begin{equation*}}
\newcommand\eeqe{\end{equation*}}
\newcommand\barr{\begin{array}}
\newcommand\earr{\end{array}}

\newcommand{\Pep}{P_{\varepsilon}}

\newcommand{\T}{{}^t}

\newcommand{\dd}{\text{d}}

\begin{document}
\renewcommand{\baselinestretch}{1.2}
%\lhead[\fancyplain{} \leftmark]{}
%\chead[]{}
%\rhead[]{\fancyplain{}\rightmark}
%\cfoot{}
%\headrulewidth=0pt
\markright{
%\hbox{\footnotesize\rm Statistica Sinica
%{\footnotesize\bf ??}(200?), 000-000}\hfill
}
\markboth{\hfill{\footnotesize\rm M. CHAMPION AND G. CHASTAING AND S. GADAT AND C. PRIEUR
}\hfill}
{\hfill {\footnotesize\rm $\mathbb L_2$-BOOSTING ON FANOVA FOR DEPENDENT INPUTS } \hfill}
\renewcommand{\thefootnote}{}
$\ $\par
\fontsize{10.95}{14pt plus.8pt minus .6pt}\selectfont
\vspace{0.8pc}
\centerline{\large\bf $\mathbb L_2$-BOOSTING ON GENERALIZED HOEFFDING}
\vspace{2pt}
\centerline{\large\bf DECOMPOSITION FOR DEPENDENT VARIABLES}
\vspace{2pt}
\centerline{\large\bf  APPLICATION TO SENSITIVITY ANALYSIS}
\vspace{.4cm}
\centerline{ Magali Champion$^{1,3}$, Gaelle Chastaing$^{1,2}$, S\'ebastien Gadat$^1$, \\ Cl\'ementine Prieur$^2$}
\vspace{.4cm}
\centerline{\em $^1$ Institut de Math\'ematiques de Toulouse }
\vspace{2pt}
\centerline{\em $^2$ Universit\'e Joseph Fourier, LJK/MOISE}
\vspace{2pt}
\centerline{\em $^3$ Institut National de la Recherche Agronomique, MIA}
\vspace{.55cm}
\fontsize{9}{11.5pt plus.8pt minus .6pt}\selectfont

\begin{quotation}
\noindent {\em Abstract:}

\par
This paper is dedicated to the study of an estimator of the generalized Hoeffding decomposition. We build such an estimator using an empirical Gram-Schmidt approach and derive a consistency rate in a large dimensional settings. Then, we apply a greedy algorithm with these previous estimators to Sensitivity Analysis. We also establish the consistency of this $\mathbb L_2$-boosting up to sparsity assumptions on the signal to analyse. We end the paper with numerical experiments, which demonstrates the low computational cost of our method as well as its efficiency on standard benchmark of Sensitivity Analysis.

\vspace{9pt}
\noindent {\em Key words and phrases:}
$\mathbb L_2$-boosting, convergence, dependent variables, generalized ANOVA decomposition, sensitivity analysis. 
\par
\end{quotation}\par

\fontsize{10.95}{14pt plus.8pt minus .6pt}\selectfont
\setcounter{chapter}{1}
\setcounter{equation}{0} %-1
\noindent {\bf 1. Introduction}

In many scientific fields, it is desirable to extend a multivariate regression model as a specific sum of increasing dimension functions. 
%Well known under the name of 
Functional ANOVA decomposition or High Dimensional Representation Model (HDMR) given by~\citet*{hooker,li} are well known expansions that allow for understanding the model behaviour, and for detecting how inputs interact to each other. For high dimensional models, the HDMR is also a good way to deal with the curse of dimensionality. Indeed, a model function may be well approximated by some first order functional components, making easier the study of a complex model. However, the existence and uniqueness of the functional ANOVA components is of major importance to valid a study. Thus, some identifiability constraints need to be imposed to make the ANOVA decomposition unique. \\
When input variables are independent, Hoeffding establishes the uniqueness of the decomposition provided that the summands are mutually orthogonal  (see \textit{e.g.} ~\citet*{hoeffding}). Further, as pointed by ~\citet*{sobol}, the analytical expression of these components can be recursively obtained in terms of conditional expectations. Thus, their estimation can be deduced by numerical approximation of integrals (see \textit{e.g} ~\citet*{sobol2,saltelliglobal}).\\
Nevertheless, the independence assumption is often unrealistic for some real-world phenomena. In this paper, we are interested in the ANOVA expansion of some models that depend on not necessarily independent input variables. Following the work of~\citet*{stone}, later exploited in machine learning by~\citet*{hooker}, and in sensitivity analysis by~\citet*{chastaing}, we focus on a generalized Hoeffding decomposition under general assumptions on the inputs distribution. That is, any model function can be uniquely decomposed as a sum of hierarchically orthogonal component functions. Two summands are called \textit{hierarchically orthogonal} whenever all variables included in one of them are also involved in the other. For a better understanding of the paper, this generalized ANOVA expansion will be called a Hierarchically Orthogonal Functional Decomposition (HOFD), as done in~\citet*{chastaing}.\\
Since analytical formulation for HOFD is rarely available, it is of great importance to develop estimation procedures.
%However, as it has been stated in~\citet*{hooker}, and showed in~\citet*{li2}, the HOFD components can not be analytically expressed when inputs are not independent. Thus, their estimation is not straightforward, and requires to use tricky numerical methods. As a first approach, Hooker~\citet*{hooker} studies the components estimation via a minimization problem under hierarchical constraints using a sample grid. However, this approach is quite computationally demanding for high-dimensional models. Moreover, no theoretical results ensure that the solution is not a local solution. In another perspective, Li \etal~\citet*{li2} propose to identify the HOFD terms by a least-squares estimation. They approximate each component as a linear combination of extended basis, and they deduce the estimation of the unknown coefficients by a continuous descent technique~\citet*{li2010}. Once again, this approach neither ensures the uniqueness of the solution, nor its consistency. \\
In this paper, we focus on an alternative method proposed in~\citet*{chastaing2} to estimate the HOFD components. It consists of constructing a hierarchically orthogonal basis from a suitable Hilbert orthonormal basis. Inspired by the usual Gram-Schmidt algorithm, the procedure recursively builds for each component a multidimensional basis that satisfies the  identifiability constraints imposed to this summand. Then, each component is well approximated on a truncated basis, where the unknown coefficients are deduced by solving an ordinary least-squares. Nevertheless, in a high-dimensional paradigm, this procedure suffers from a curse of dimensionality. Moreover,  it is numerically observed that only a few of coefficients are not close to zero, meaning that only a small number of predictors restore the major part of the information contained in the components. Thus, it is important to be able to select the most relevant representative functions, and next identify the HOFD with a limited computational budget.\\
In this view, we suggest in this article to transform the ordinary least-squares into a penalized regression as it has been proposed in~\citet*{chastaing2}. In the present paper, we focus here on the $\mathbb L_2$-boosting to deal with the $\ell_0$ penalization, developped by~\citet*{friedman}. The $\mathbb L_2$-boosting is a greedy strategy that performs  variable selection and shrinkage. The choice of such an algorithm is motivated by the fact that the $\mathbb L_2$-boosting is very intuitive and easy to implement. It is also closely related (in some practical sense) to the LARS algorithm, proposed by~\citet*{efron}, which solves the Lasso regression with a $\ell_1$ penalization (see \textit{e.g.}~\citet*{buhlmann,tibshirani}). The $\mathbb L_2$-boosting and the LARS both select predictors using the maximal correlation with the current residuals.
%It is also closely related to the $\ell_1$-penalized Lasso, that offers a statistical accuracy compared with the usual greedy algorithms~\citet*{buhlmann,tibshirani}.
 The question that naturally arises now is the following: provided that the theoretical procedure of components reconstruction is well tailored, do the estimators obtained by the $\mathbb L_2$-boosting converge to the theoretical true sparse parameters when the number of observations tends to infinity ?\\
The goal of this paper is to extend the work of~\citet*{chastaing2} by addressing this question. More precisely, the aim is to determine sufficient conditions for which the consistency of the estimators is satisfied. Further, we discuss these conditions and  give some numerical examples where such conditiones are fulfilled. 
One interesting application of the general theory is the global sensitivity analysis (SA). We apply the $\mathbb L_2$-boosting to estimate the generalized sensitivity indices defined in~\citet*{chastaing,chastaing2}. After reminding the form of these indices, we  numerically  compare the $\mathbb L_2$-boosting performance with the LARS technique and the Forward-Backward algorithm, proposed by~\citet*{zhang}. %The goal is also to show that the $\mathbb L_2$-boosting brings very satisfying numerical results, with a low computational cost in comparison with the other two penalized strategies.

The article is organized as follows. Paragraph 2.1 aims at introducing the notation of the paper.% describing the notation that will be used further along the article. 
We also remind the HOFD representation of the model function in Paragraph 2.2. In Paragraph 2.3, we recall the procedure detailed in~\citet*{chastaing2} that consists in constructing well tailored hierarchically orthogonal basis to represent the components of the HOFD. At last, we highlight the curse of dimensionality we are exposed to, and present the $\mathbb L_2$-boosting. Section 3 gathers our main theoretical results on the proposed algorithms. Section 4 presents a numerical study of our method. We finally conclude this work in Section 5, and we provide the proofs of the two main theorems in an Appendix.\\

\noindent {\large\bf Acknowledgment}
Authors are indebted to Fabrice Gamboa for motivating discussions and numerous suggestions on the subject.
\par

\setcounter{chapter}{2}
\setcounter{equation}{0} %-1
\noindent {\bf 2. Estimation of the generalized Hoeffding decomposition components}\label{sec:notation}
%\section{Estimation of the generalized Hoeffding decomposition components}\label{sec:notation}

{\bf 2.1 Notation}\label{sec:not}

%\section{Notation}\label{sec:not}
We consider a measurable function $f$ of a random real vector $\XX=(X_1,\cdots,X_p)$ of $\mathbb R^p$, $p\geq 1$. The response variable $Y$ is a real-valued random variable defined as

\begin{equation}\label{eq:model}
Y=f(\XX) + \varepsilon,
\end{equation}
where $\varepsilon$ stands for a centered random variable independent of $\XX$ and models the variability of the response around its theoretical unknown value $f$.
We denote by $\PX$ the distribution law of $\XX$, which is unknown in our setting, and we assume that $\XX$ admits a density function $\pX$ with respect to the Lebesgue measure on $\mathbb R^p$. Note that $\PX$ is not necessarily a tensor product of univariate distributions since the components of $X$ may be correlated.% since we are primarily interested in the context of not necessarily independent inputs. 

Further, we suppose that $f\in L^2_{\mathbb R}(\mathbb R^p,\mathcal B(\mathbb R^p),\PX)$, where $\mathcal B(\mathbb R^p)$ denotes the Borel set of $\mathbb R^p$. The Hilbert space $L^2_{\mathbb R}(\mathbb R^p,\mathcal{B}(\mathbb{R}^p),\PX)$ is denoted by $L^2_{\mathbb R}$, for which we use the inner product $\psth{\cdot}{\cdot}$, and the norm $\norm{\cdot}$ as follows,

\[
 \langle h,g\rangle=\int h(\mathbf x)g(\mathbf x) \pX d\xx = \mathbb E(h(\mathbf X)g(\mathbf X))
\]
\[
\norm{h}^2=\psth{h}{h}=\E(h(\XX)^2), \quad \forall h,g\in L_{\mathbb R}^2. 
\]
Here, $\E(\cdot)$ stands for the expected value. Further, $V(\cdot)=\mathbb E[(\cdot-\mathbb E(\cdot))^2]$ denotes the variance,
and $\mbox{Cov}(\cdot,\ast)=\mathbb E[(\cdot-\mathbb E(\cdot))(\ast-\mathbb E(\ast))]$ the covariance.

%Since  we will need to consider individual effects of the considered variables, 
For any $1 \leq i \leq p$, we denote by $\PXi$ the marginal distribution of $\XXi$ and extend naturally the former notation 
to $L^2_{\mathbb R}(\mathbb R,\mathcal B(\mathbb R),\PXi) := \LRi$.\\

\par

{\bf 2.2 The generalized Hoeffding decomposition}\label{sec:HOFD}

Let us denote $\disc 1 k:=\{1,2,\cdots,k\}$, with $k \in \mathbb N^*$, and let $S$ be the collection of all subsets of $\disc 1 p$. We also define $S^*:=S\setminus \{ \emptyset \}$. For $u\in S$, the subvector $\XuXu$ of $\XX$ is defined as $\XuXu:=(X_i)_{i\in u}$. 
%The complementary vector of $\XuXu$ is denoted $\XX_{-u}$. 
Conventionally, for $u=\emptyset$, $\XuXu=1$. The marginal distribution (\textit{resp.} density) of $\XuXu$ is denoted $P_{\XuXu}$ (\textit{resp.} $\pXu$). 

%{\color{red}
%The aim of \textcolor{blue}{'the'?} functional ANOVA is to decompose $f$ as a sum of increasing dimension functions, 
A functional ANOVA decomposition consists in expanding $f$ as a sum of increasing dimension functions,
%Under mild assumptions on the joint distribution $\pX$, it is shown in~\cite{chastaing} that the theoretical model function $f$ can be uniquely decomposed as a sum of increasing dimension functions, 
\beq\label{fanova}
\barr{lll}
f(\XX)&=&f_\emptyset+ \sum_{i=1}^p f_i(X_i) + \sum_{1\leq i <j \leq p} f_{ij}(X_i,X_j)+ \cdots+ f_{1,\cdots,p}(\XX)\\
&=& \sum_{u \in S} f_u(\XuXu),
\earr
\eeq
where $f_\emptyset$ is a constant term, $f_i$, $i \in \disc 1 p$ are the main effects, $f_{ij}, f_{ijk}, \cdots$, $i,j,k \in \disc 1 p$ are the interaction effects, and the last component $f_{1,\cdots,p}$ is the residual. 

Decomposition \eqref{fanova} is generally not unique. However, under mild assumptions on the joint density $\pX$ (see Assumptions (C.1) and (C.2) in \cite{chastaing}), the decomposition is unique under some additional orthogonality assumptions.

 More precisely, let us introduce $H_{\emptyset}=H_{\emptyset}^0$ the set of constant functions, and for all $u \in S^*$, $H_u:=L^2_{\mathbb R}(\mathbb R^u,\mathcal{B}(\mathbb{R}^u),P_{\XuXu})$. Then we define
$H_u^0$, $u \in S\setminus \emptyset$ as follows:
\[
 H_u^0=\left\{ h_u \in H_u, ~\langle h_u,h_v\rangle=0, \forall~v \subset u, \forall~h_v \in H_v^0 \right\},
\]
where $\subset$ denotes the strict inclusion.

\begin{defi}[Hierarchical Orthogonal Functional Decomposition - HOFD]\label{def1}
Under Assumption (C.1) and (C.2) in \cite{chastaing}, the decomposition (\ref{fanova}) is unique as soon as we assume $f_u \in H_u^0$ for all $u \in S$. 
\end{defi}
\begin{rem}\label{rem1}
The components of the HOFD \eqref{fanova} are referred as hierarchically orthogonal, that is $\langle f_u,f_v \rangle =0$ $\forall v \subset u$. 
\end{rem}
 To get more details on the HOFD, the reader is referred to~\cite{hooker,chastaing}. 
In this paper, we are interested in estimating the summands in (\ref{fanova}). 
As underlined in~\cite{huang}, estimating all components of (\ref{fanova}) suffers from a curse of dimensionality, leading to an intractable problem in practice. To bypass this issue, we assume further along the article (without loss of generality) that $f$ is centered, so that $f_\emptyset=0$ and suppose that $f$ is well approximated by

\begin{equation}\label{approxdec}
f(\XX)\simeq\sum_{\substack{u\in S^*\\ |u|\leq d}} f_u(\XuXu), \quad d\ll p
\end{equation}  

We thus assume that interactions of order $\geq d+1$ can be neglected. 
But even by choosing $d=2$, the number of components in  (\ref{approxdec}) can become prohibitive if the number of inputs $p$ is high. We therefore are interested by estimation procedures under sparse assumptions when the number of variables $p$ is large.\\
In the next section, we remind the procedure to identify components of  (\ref{approxdec}). Through this strategy, we highlight the curse of dimensionality when $p$ is getting large, and we propose to use a greedy $\mathbb L_2$-boosting to tackle this issue.\\

{\bf 2.3 Practical determination of the Sparse HOFD}\label{sec:algo}

{\bf General description of the procedure}

%\subsubsection*{General description of the procedure}
%{\color{red}
We propose in this section a Two-Steps estimation procedure to identify the components in (\ref{approxdec}): the first one is a simplified version of the Hierarchical Orthogonal Gram-Schmidt (HOGS) procedure developed in \cite{chastaing2}, and the second consists of a $\mathbb L_2$-boosting algorithm (see \textit{e.g.} \cite{friedman,buhlmann06}). The specificity of our new $\mathbb L_2$-boosting algorithm is that it is based on a random dictionary and then falls into the framework of sparse recovery problem with error in the variables.

To lead this two-steps procedure, we assume that we observe two independent and identically distributed samples $(y^r,\xx^r)_{r=1,\cdots,n_1}$ and $(y^s,\xx^s)_{s=1,\cdots,n_2}$ from the distribution of $(Y,\XX)$ (the initial sample can be splitted in such two samples). We define the empirical inner product  $\psemp{\cdot}{\cdot}$ and the empirical norm $\normemp{\cdot}$ associated to a $n$-sample as

\beqe
\psemp{h}{g}=\frac 1 n \sum_{s=1}^n h(\xx^s)g(\xx^s),\quad \normemp{h}=\psemp{h}{h}.
\eeqe
Also, for $u=(u_1,\cdots, u_t)\in S$, we define the multi-index $\lu=(l_{u_1},\cdots, l_{u_t})\in \mathbb N ^t$. We use the notation $\gener{B}$ to define the set of all finite linear combination of elements of $B$, also called the linear span of $B$. 

Step 1 and  Step 2 of our sparse HOFD procedure will be described in details further below.

%{\color{red}
\begin{rem}\label{rem2}
In the following, we assume that $d=2$ in (\ref{approxdec}). The procedure could be extended to any higher order approximation, but we think that the description of the methodology for $d=2$ helps for a better understanding. We thus have chosen to only describe this situation for the sake of clarity.
\end{rem}
%}

{\bf Step 1: Hierarchically Orthogonal Gram-Schmidt procedure}\label{par:HOFD}

%\subsubsection*{Step 1: Hierarchically Orthogonal Gram-Schmidt procedure}
For each $i \in [1:p]$, let $\{\Psi_{l_i}^i \, , \; l_i \in \mathbb{N}\}$ denote an orthonormal basis of $H_i:=L^2(\mathbb R,\mathcal B(\mathbb R),P_{X_i})$. For $L \in \mathbb N^*$, for $i\neq j \in \disc 1 p$, we set
\[
H_{\emptyset}^L=\gener{1} \quad \text{and} \quad 
H_i^L= \gener{1,\psi_{1}^i,\cdots,\psi_{L}^i},
\]
as well as
\[
 H_{ij}^{L}= \gener{ 1,\psi_{1}^i,\cdots,\psi_{L}^i,\psi_{1}^j,\cdots,\psi_{L}^j, \psi_{1}^i\otimes \psi_{1}^j ,\cdots,\psi_{L}^i \otimes \psi_{L}^j}.
\]
We define $H_u^{L,0}$, the approximation of $H_u^0$, as
\[
 H_u^{L,0}=\left\{ h_u \in H_u^L, ~\langle h_u,h_v\rangle=0, \forall~v \subset u, \forall~h_v \in H_v^{L,0} \right\},
\]
The recursive procedure below aims at constructing a basis of $H_i^{L,0}$ and a basis of $H_{ij}^{L,0}$ for any $i \neq j \in [1:p]$.

%
%We briefly recall in this paragraph the HOFD that yields $(H_u)_{u \in S}$.
% As pointed above, the construction of such a sequence is recursive and use a baseline truncated orthonormal system $(\pfi)_{l_i=0}^L$ of $\LRi$ (with the convention $\pf{0}{i}=1$). In the sequel, we will omit the index $L$ in $H_u^L$ for sake of convience and just use  $H_u$.

\paragraph{Initialization}
For any $1 \leq i \leq p$, define $\phi_{l_i}^i:=\Psi_{l_i}^i$, $l_i \in \disc 1 L$. Then, thanks to the orthogonality of $\{\Psi_{l_i}^i \, , \; l_i \in \mathbb{N}\}$, we get
$H_i^{L,0}:=\gener{\pf 1 i, \cdots \pf L i}.
$
\paragraph{Second order interactions}
Let $u=\{i,j\}$, with $i\neq  j\in \disc 1 p$. As the dimension of $H_{ij}^L$ is equal to $L^2+2L+1$, and that the approximation space $H_{ij}^{L,0}$ is subject to $2L+1$ constraints, its dimension is then equal to $L^2$. We want to construct a basis for $H_{ij}^{L,0}$, which satisfies the hierarchical orthogonal constraints. We are looking for such a basis of the form:

\beq\label{extend}
\barr{lll}
\pfij(X_i,X_j)&=&\pfi(X_i)\times \pfj(X_j)+\sum_{k=1}^L \lambda_{k,\bolds{\bolds{l_{ij}}}}^i \pf k i(X_i)\\
&&+ \sum_{k=1}^L \lambda_{k,\bolds{\bolds{l_{ij}}}}^j \pf k j(X_j)+C_{\bolds{\bolds{l_{ij}}}},
\earr
\eeq
with $\bolds{\bolds{l_{ij}}}=(l_i,l_j)\in \disc 1 L^2$.

\noindent  The constants $(C_{\bolds{\bolds{l_{ij}}}},(\lambda_{k,\bolds{\bolds{l_{ij}}}}^i)_{k=1}^L,(\lambda_{k,\bolds{\bolds{l_{ij}}}}^j)_{k=1}^L) $ are determined by resolving the following constraints:
\beq\label{basiscond}
\barr{l}
\psth{\pfij}{\pf k i}=0,\quad \forall~k \in \disc 1 L\\
\psth{\pfij}{\pf k j}=0,\quad \forall~k \in \disc 1 L\\
\psth{\pfij}{1}=0.
\earr
\eeq

\noindent  We first solve the linear system:
\beq\label{systmat}
A^{ij} \inc^{\bolds{l_{ij}}}=D^{\bolds{l_{ij}}},
\eeq
where $A^{ij}=
 \begin{pmatrix}
 \mathbb{E}(\Phi_i{}^t\Phi_i) & \mathbb{E}(\Phi_i{}^t\Phi_j) \\
 \mathbb{E}(\Phi_j{}^t\Phi_i) & \mathbb{E}(\Phi_j{}^t\Phi_j)
 \end{pmatrix}$, with $(\Phi_i)_{k}=\phi_{k}^i$, and $(\Phi_j)_{k}=\phi_{k}^j$ for $ k \in \disc 1 {L}$. Also, 
%\beqe
%A^{ij}=
% \begin{pmatrix}
% \mathbb{E}(\Phi_i{}^t\Phi_i) & \mathbb{E}(\Phi_i{}^t\Phi_j) \\
% \mathbb{E}(\Phi_j{}^t\Phi_i) & \mathbb{E}(\Phi_j{}^t\Phi_j)
% \end{pmatrix}, 
% \begin{array}{ll}
%  (\Phi_i)_{k}=\phi_{k}^i, & k \in \disc 1 {L}, \\
%  (\Phi_j)_{k}=\phi_{k}^j,  & k \in \disc 1 {L},
%  \end{array}
%\eeqe 
$\inc^{\bolds{l_{ij}}} =\begin{pmatrix} \lambda_{1,\bolds{l_{ij}}}^i & 
\cdots &
\lambda_{L,\bolds{l_{ij}}}^i&
\lambda_{1,\bolds{l_{ij}}}^j&
\cdots&
\lambda_{L,\bolds{l_{ij}}}^j\end{pmatrix}\T$, \\
$D^{\bolds{l_{ij}}}= -
\begin{pmatrix}
\psth{\phi_{l_i}^i\times \phi_{l_j}^j}{\phi_1^i}&
\cdots&
\psth{\phi_{l_i}^i\times \phi_{l_j}^j}{\phi_{L}^i}&
\psth{\phi_{l_i}^i\times \phi_{l_j}^j}{\phi_1^j}&
\cdots&
\psth{\phi_{l_i}^i\times \phi_{l_j}^j}{\phi_{L}^j}
\end{pmatrix}\T$.\\
%\beqe
%\inc^{\bolds{l_{ij}}} = 
%\begin{pmatrix}
%\lambda_{1,\bolds{l_{ij}}}^i\\
%\vdots\\
%\lambda_{L,\bolds{l_{ij}}}^i\\
%\lambda_{1,\bolds{l_{ij}}}^j\\
%\vdots\\
%\lambda_{L,\bolds{l_{ij}}}^j
%\end{pmatrix},~~ 
%D^{\bolds{l_{ij}}}= -
%\begin{pmatrix}
%\psth{\phi_{l_i}^i\times \phi_{l_j}^j}{\phi_1^i}\\
%\vdots\\
%\psth{\phi_{l_i}^i\times \phi_{l_j}^j}{\phi_{L}^i}\\
%\psth{\phi_{l_i}^i\times \phi_{l_j}^j}{\phi_1^j}\\
%\vdots\\
%\psth{\phi_{l_i}^i\times \phi_{l_j}^j}{\phi_{L}^j}
%\end{pmatrix}.
%\eeqe
As shown in~\cite{chastaing2}, $A^{\bolds{l_{ij}}}$ is a definite positive Gramian matrix and (\ref{systmat}) admits a unique solution in $\inc^{\bolds{l_{ij}}}$. Next, $C_{\bolds{l_{ij}}}$ is deduced with
\begin{equation}\label{cresoudre}
 C_{\bolds{l_{ij}}}=-\E\bigg[ \phi_{l_i}^i\otimes \phi_{l_j}^j(X_i,X_j)+\sum_{k=1}^{L}\lambda_{k,\bolds{l_{ij}}}^i\phi_{k}^i(X_i) +\sum_{k=1}^{L}\lambda_{k,\bolds{l_{ij}}}^j\phi_{k}^j(X_j) \bigg].
 \end{equation}
 
%Once we have proceeded in the same way for any multi-index $\bolds{l_{ij}}$, we obtain $L^2$ representative functions:
%  \beqe
% H_{ij}=\gener{\phi_1^{ij}, \cdots, \phi_{\bolds{l_{ij}}}^{ij}}, \quad \bolds{l_{ij}}=L^2
% \eeqe

 \paragraph{Higher interactions}
 This construction can be extended to any $|u|  \geq 3$. We refer the interested reader to \cite{chastaing2}. Just note that the dimension of the approximation space $H_u^{L,0}$ is given by $L_u=L^{|u|}$, where $|u|$ denotes the cardinality of $u$.

\paragraph{Empirical procedure}

Algorithm \ref{EHOFD} below proposes an empirical version of the HOGS procedure. It consists in
substituting the inner product $\psth \cdot \cdot$ by its empirical version $\psempn \cdot \cdot$ obtained with the first data set  $(y^r, \xx^r)_{r=1,\cdots,n_1}$.

 \begin{algorithm}[H]
 \caption{Empirical HOFD (EHOFD)}\label{EHOFD}
 \BlankLine
\KwIn{Orthonormal system $(\pfi)_{l_i=0}^L$ of $H_i$, $i \in [1:p]$, i.i.d. observations ${\cal O}_1 :=(y^r, \xx^r)_{r=1,\cdots,n_1}$ of \eqref{eq:model}, threshold $|u_{max}|$}
\emph{Initialization:}  for any $i\in \disc 1 p$ and $l_i\in \disc 1 L$, define first $\vfi=\pfi$.

\begin{itemize}
\item  For any $u$ such that $2 \leq |u|\leq |u_{max}|$, write the matrix $(\hat{A}^{ij}_{n_1})$ as well as $(\hat{D}^{\bolds{l_{ij}}}_{n_1})$ obtained  using the former expressions with $\psempn \cdot \cdot$.

\item Solve (\ref{systmat}) with the empirical inner product  $\psempn \cdot \cdot$ and compute $(\bolds{\hat{\lambda}}^{\bolds{l_{ij}}}_{n_1})$.

\item Compute $ \hat C_{\bolds{l_{ij}}}^{n_1}$ by using Equation (\ref{cresoudre}) and $(\hat{\bolds{\lambda}}_{n_1}^{\bolds{l_{ij}}})$.

\item The empirical version of the basis given by \eqref{extend} is then:
\beqe
\forall u \in [2: |u_{max}|]  \quad 
\hat{H}^{L,0,n_1}_u=\gener{\vf 1 u,\cdots, \vf {L_u} u}, \, \text{where}\, ~ L_u=L^{|u|}.
\eeqe
\end{itemize}
 \end{algorithm}

 {\bf Step 2: Greedy selection of Sparse HOFD}
 
% \subsubsection*{Step 2: Greedy selection of Sparse HOFD}
Each component $f_u$ of the HOFD defined in Definition \ref{def1} is a projection onto $H_u^0$. Since, for $u \in S^*$, the space $\hat{H}^{L,0,n_1}_u$ well approximates $H_u^0$, it is then natural to approximate $f$ by:
\beqe
 f(\xx)\simeq \bar f(\xx)=\sum_{\substack{u\in S^*\\ |u|\leq d}}\bar f_u(\xuxu),   ~\textrm{with }~\bar f_u(\xuxu)= \sum_{\lu}\beta_{\lu}^u \vfu(\xuxu),
\eeqe
where $\lu$ is the multi-index $\lu=(l_i)_{i\in u} \in \disc 1 L ^{|u|}$. For the sake of clarity (since there is no ambiguity), we will omit the summation support of $\lu$ in the sequel.

Now, we consider the second sample $(y^s, \xx^s)_{s=1,\cdots,n_2}$ and we aim to recover the unknown coefficients $(\beta_{\lu}^u)_{\lu,|u|\leq d}$ on the regression problem,

\beqe
y^s= \bar f(\xx^s)+ \varepsilon^s, \quad s=1,\cdots, n_2.
\eeqe

However, the number of coefficients is equal to $\sum_{k=1}^d \binom p k L^k$. When $p$ gets large, the usual least-squares estimator is not adapted to estimate the coefficients $(\beta_{\lu}^u)_{\lu,u}$. We then use the penalized regression,
\beqe
(\hat\beta_{\lu}^u) \in \argmin_{\beta_{\lu}^ u \in \mathbb R} \frac 1 {n_2} \sum_{s=1}^{n_2} \bigg[  y^s- \sum_{\substack{u\in S^*\\ |u|\leq d}} \sum_{\lu}\beta_{\lu}^u \vfu(\xuxu^s)  \bigg]^2+ \lambda J(\beta_1^1,\cdots, \beta_{\lu}^ u,\cdots),
\eeqe
where $J(\cdot)$ is the $\ell_0$-penalty, i.e. 
$$ J(\beta_1^1,\cdots, \beta_{\lu}^ u,\cdots)=\sum_{\substack{u\in S^*\\ |u|\leq d}} \sum_{\lu} \mathds 1(\beta_{\lu}^ u\neq 0).
$$ 

Of course, such an optimisation procedure is not tractable and we instead consider the relaxed $\mathbb L_2$-boosting (see \textit{e.g.}~\cite{friedman}) to solve this penalized problem. Mimicking the notation of~\cite{temlyakov00, champion}, we define the dictionary $\mathcal D$ of functions as
\beqe
\mathcal D=\{ \vf 1 1, \cdots \vf L 1,\cdots, \vf 1 u,\cdots, \vf {L_u} u,\cdots \}.
\eeqe
The quantity $G_k(\bar f)$  denotes the approximation of $\bar f$ at step $k$, as a linear combination of elements of $\mathcal D$. At the end of the algorithm, the estimation of $\bar f$ is denoted $\hat f$. The $\mathbb L_2$-boosting  is described in Algorithm \ref{algoboost}.

 \begin{algorithm}[H]
 \caption{The $\mathbb L_2$-boosting}\label{algoboost}
 \BlankLine
\KwIn{ Observations ${\cal O}_2 := (y^s,\xx^s)_{s=1,\cdots,n_2}$, shrinkage parameters $\gamma \in ]0,1]$ and  number of iterations $k_{up}\in \mathbb N^*$.}

\emph{\textbf{Initialization}}: $ G_0(\bar f)=0$. \\

\For{$k=1$ to $k_{up}$}{
\begin{enumerate}
\item Select $\vfs \in \mathcal D$ such that
\beq\label{eqphi2}
|\psempnn {Y- G_{k-1}(\bar f)}{\vfs} | = \max_{\vfu \in \mathcal D} |\psempnn { Y- G_{k-1}(\bar f)}{\vfu} |.
\eeq
\item Compute the new approximation of $\bar f$ as 

\beq\label{hatg}
 G_k(\bar f)= G_{k-1}(\bar f)+\gamma \psempnn { Y- G_{k-1}(\bar f)}{\vfs} \cdot \vfs.
\eeq
\end{enumerate}
}
\KwOut{$\hat f= G_{k_{up}}(\bar f)$.}

 \end{algorithm}
 
For any step $k$, Algorithm \ref{algoboost} selects a function from $\mathcal D$ wich provides a sufficient information on the residual
$Y- G_{k-1}(\bar f)$. The shrinkage parameter $\gamma$ is the standard step-length parameter of the boosting algorithm. It actually smoothly inserts the next predictor in the model, making possible a refinement of the greedy algorithm, and may statistically guarantees its  convergence rate. 
\begin{rem}
In a deterministic setting, the shrinkage parameter is not really useful and may be set to $1$ (see \cite{temlyakov00} for further details). It is indeed useful from a practical point of view to smooth the boosting iterations. % Notice also that we provide here a simplified version of the $\mathbb L_2$-boosting, as the selection of the predictor in (\ref{eqphi2}) depends on an another parameter $\nu$. We assume here it is fixed to $\nu=1$.
\end{rem}
{\bf An algorithm for our new sparse HOFD procedure}

%\subsubsection*{An algorithm for our new sparse HOFD procedure}
Algorithm \ref{GHOFD}  below provides now a simplified description of our sparse HOFD procedure, whose steps have been described further above.

 \begin{algorithm}[H]
 \caption{Greedy Hierarchically Orthogonal Functional Decomposition}\label{GHOFD}
 \BlankLine
\KwIn{Orthonormal system $(\Psi_{l_i}^i)_{l_i=0}^L$ of $L^2(\mathbb R,\mathcal B(\mathbb R),P_{X_i})$, $i \in [1:p]$, i.i.d.  observations ${\cal O} := (y^j,\xx^j)_{j=1 \ldots n}$ of \eqref{eq:model}}
\emph{Initialization:} Split ${\cal O}$ in a partition ${\cal O}_1 \cup {\cal O}_2$ of size $(n_1,n_2)$.

\begin{itemize}
\item For any $u \in S$, use Step 1 with observations ${\cal O}_1$ to construct the approximation $\hat{H}_u ^{L,0,n_1}  := \gener{\vf 1 u,\cdots, \vf {L_u} u}$ of $H_u^{L,0}$ (see Algorithm \ref{EHOFD}).
\item Use an $\mathbb L_2$-boosting algorithm on ${\cal O}_2$ with the random dictionary $\mathcal D=\{ \vf 1 1, \cdots \vf L 1,\cdots, \vf 1 u,\cdots, \vf {L_u} u,\cdots \}$ to obtain the Sparse Hierarchically Orthogonal Decomposition  (see Algorithm \ref{algoboost}).
\end{itemize}
 \end{algorithm}
 
 We now obtain a strategy to estimate the components of the decomposition (\ref{approxdec}) in a high-dimensional paradigm. We aim to show that the obtained estimators are consistent, and that the Two-Steps procedure (summarized in Algorithm \ref{GHOFD}) is numerically convincing. The next section is devoted to the asymptotic properties of the estimators. \\
\par

\setcounter{chapter}{3}
\setcounter{equation}{0} %-1
\noindent {\bf 3. Consistency of the estimator}\label{sec:main}

In this section, we study the asymptotic properties of the estimator $\hat f$ obtained from the Algorithm \ref{GHOFD} described in Section 2. To this end, we restrict our study  to the case of $d=2$ and assume that $f$ is well approximated by first and second order interaction components. Hence, the observed signal $Y$ may be represented as

\beqe
Y=\sum_{\substack{u\in S^*\\ |u|\leq 2}} \sum_{\bolds{l_u}}\beta_{\bolds{l_u}}^{u,0} \pfu(\XuXu)+\varepsilon, \quad \E(\varepsilon)=0,~ \E(\varepsilon^2)=\sigma^2,
\eeqe
where $\bolds\beta^0=(\beta_{\bolds{l_u}}^{u,0})_{\bolds{l_u},u}$ is the true parameter, and the functions $(\pfu)_{\bolds{l_u}}$, $|u|\leq 2$ are constructed according to the HOFD described in the paragraph \ref{par:HOFD}.
We assume that we have in hand a $n$-sample of observations, divided into two samples ${\cal O}_1$ and ${\cal O}_2$. Samples in ${\cal O}_1$ (\textit{\textit{resp.}} in ${\cal O}_2$) of size $n_1=n/2$ (\textit{\textit{resp.}} of size $n_2=n/2$) are used for the construction of $(\vfu)_{\bolds{l_u},u}$ described in Algorithm \ref{EHOFD} (\textit{\textit{resp.}} for the $\mathbb L_2$-boosting Algorithm \ref{algoboost} to estimate $(\beta_{\bolds{l_u}}^u)_{\bolds{l_u},u}$).

 The goal of this section  is to study the consistency of $\hat f=G_{k_n}(\bar f)$ when the sample size $n$ tends to infinity. Its objective is also to determine an optimal number of steps $k_{n}$ to get a consistent estimator from Algorithm \ref{algoboost}. \\

{\bf 3.1 Assumptions}

We first briefly recall some notation: for any sequences $(a_n)_{n\geq 0}$, $(b_n)_{n\geq 0}$, we write $a_n=\underset{n\rightarrow +\infty}{\mathcal O} (b_n)$ when $a_n/b_n$ is a bounded sequence for $n$ large enough. Now, for any random sequence $(X_n)_{n \geq 0}$, $X_n=\gop(a_n)$ means that $\abs{X_n/a_n}$ is bounded in probability. 

We have chosen to present our assumptions in three parts to deal with the dimension, the noise and the sparseness of the entries.

\paragraph{Bounded Assumptions $\mathbf{(H_{b})}$} The first set of hypotheses matches with the \textit{bounded case} and is adapted to the special situation of bounded support for the random variable $X$, for instance when each $X_j$ follows a uniform law on a compact set $\mathcal{K}_j \subset K$ where $K$ is a compact set of $\mathbb{R}$ independent of $j \in [1:p	]$. It is refered as $\mathbf{(H_{b})}$ in the sequel and corresponds to the following three conditions.

 \begin{enumerate}
 \item[] $\mathbf{(H_{b}^1)}$ $M:=\sup_{\substack{i\in \disc 1 p\\ l_i\in \disc 1 {L}}}\norm{\pfi(X_i)}_{\infty}<+ \infty$, 
\item[]  $\mathbf{(H_{b}^2)}$ The number of variables $p_n$ satisfies $$p_n=\gophyp, \text{ where } 0<\xi \leq 1  \text{ and } C>0.$$
\item[]  $\mathbf{(H_{b}^{3,\vartheta})}$ The Gram matrices $A^{ij}$ introduced in \eqref{systmat} satisfies:
$$
\exists C>0 \,  \quad \forall (i,j) \in [1:p_n]^2 \qquad det(A^{ij}) \geq C n^{- \vartheta},
$$
where $det$ denotes the determinant of a matrix.
\end{enumerate} 

 Roughly speaking, this will be the favorable situation from a technical point of view since it will be possible to apply a Matrix Hoeffding's type Inequality. 
It may be possible to slightly relax such an hypothesis using a sub-exponential tail argument. For the sake of simplicity, we have chosen to only restrict our work to the settings of $\mathbf{(H_{b})}$.

Whatever the joint law of the random variables $(X_1,\ldots,X_p)$ is, it is always possible to build an orthonormal basis $(\pfi)_{1 \leq l_i \leq L}$ from a bounded (frequency truncated) Fourier basis and thus $\mathbf{(H_{b}^1)}$ is not so restrictive in practice.

Assumption $\mathbf{(H_{b}^2)}$ copes with the high dimensional situation. The number of variables $p_n$ can grow exponentially fast with the number of observations $n$.

Note that Hypothesis $\mathbf{(H_{b}^{3,\vartheta})}$ stands for a lower bound of the determinant of the Gram matrices involved in the HOFD. It is shown in \citet*{chastaing2} that each of these Gram matrices are invertible and thus each $\det(A^{ij})$ are positive. Nevertheless, if $\vartheta = 0$, this hypothesis  assume that such an invertibility is \textit{uniform} over all choices of tensor $(i,j)$. 
This hypothesis may be too strong for a large number of variables $p_n \rightarrow + \infty$ when $\vartheta=0$. However, when $\vartheta>0$, Hypothesis $\mathbf{(H_{b}^{3,\vartheta})}$ drastically relax the case $\vartheta=0$ and becomes very weak. It will be satisfied in many of our numerical examples. In the sequel, the parameters $\vartheta$ and $\xi$ will be related each other and 
we will obtain a consistency result of the sparse HOFD up to the condition $\vartheta < \xi/2$.
%$\xi \geq 2 \vartheta$. 
This constraint implicitely limits the size of $p_n$ since  $ \log p_n = \underset{n\rightarrow + \infty}{\mathcal O}(n^{1-\xi})$.

\paragraph{Noise Assumption $\mathbf{(H_{\varepsilon,q})}$}

We will assume the noise measurement $\varepsilon$ to get some bounded moments of sufficiently high order, which is true for Gaussian or bounded noise. This assumption is given by

\begin{enumerate}
\item[] $\mathbf{(H_{\varepsilon,q})}$ $\E(|\varepsilon|^q) < \infty, \quad \text{for one } q \in \mathbb{R}_+.$
\end{enumerate}

\paragraph{Sparsity Assumption $\mathbf{(H_{s})}$}

The last assumption concerns the sparse representation of the unknown signal described by $Y$ in the basis $(\phi_{\bolds{l_u}}^u(\XuXu))_{u}$. Such an hypothesis will be usefull to assess the statistical performance of the $\mathbb L_2$-boosting and will be refered as $\mathbf{(H_{s})}$ in the sequel. It is legitimate by our high dimension setting and our motivation to identify the main interactions $\XuXu$.

\begin{enumerate}
\item[] $
\mathbf{(H_{s})}$  The true parameter  $\bolds\beta^0$  satisfies  uniformly with $n$  $$\bolds\|\beta^0\|_{L^1}:=\betasum < \infty.$$
\end{enumerate}

It is possible to relax this former condition and let $\| \bolds\beta^0\|_{L^1}$ growing to $+\infty$ as $n \rightarrow + \infty$. The price to pay to face such a situation is then a more restrictive condition on the number of variables $p_n$. We  refer to \citet*{buhlmann06} for a short discussion on a related problem and will only consider the situation described by $\mathbf{(H_{s})} $ for the sake of simplicity.\\

{\bf 3.2 Main results}

We first provide our main result on the efficiency of the EHOFD (Algorithm \ref{EHOFD}).
\begin{theor}\label{theo:base} Assume that  $\mathbf{(H_{b})}$ holds with $\xi$ (\textit{resp.} $\vartheta$) given by $\mathbf{(H_{b}^2)}$ (\textit{resp.} $\mathbf{(H_{b}^{3,\vartheta})}$). Then, if $\vartheta< \xi/2$, the sequence of estimators $(\vfu)_{u}$ satisfies:
$$
 \sup_{\substack{u \in S^*, |u|\leq d\\ \bolds{l_u} }} \norm{\vfu-\pfu}=\zeta_{n,0}=\gopxiv.
$$
\end{theor}

The proof of this Theorem is deferred to the Appendix section.
Our second main result concerns the $\mathbb L_2$-boosting which recovers the unknown $\tilde{f}$ up to a preprocessing estimation of  $(\vfu)_{\bolds{l_u},u}$ on a first sample ${\cal O}_1$. Such a result is satisfied provided the sparsity Assumptions $\mathbf{(H_{s})}$.
We assume that 
\[
Y=\tilde f(\XX)+\varepsilon, \quad \tilde f(\XX)=\sum_{\substack{u\in S^*\\ |u|\leq d}} \sum_{\bolds{l_u}}\beta_{\bolds{l_u}}^{u,0} \pfu(\XuXu) \in H_u^L,
\]
where $\bolds\beta^0=(\beta_{\bolds{l_u}}^{u,0})_{\bolds{l_u},u}$ is the true parameter that expands $\tilde{f}$.

\begin{theor}[Consistency of the $\mathbb L_2$-boosting]\label{thcvboost}
 Consider an estimation $\hat f$ of $\tilde  f$ from an i.i.d. $n$-sample broken up into ${\cal O}_1 \cup {\cal O}_2$. 
Assume  that functions $(\vfu)_{\bolds{l_u},u}$ are estimated from the first sample ${\cal O}_1 $ under  $\mathbf{(H_{b})}$ with $\vartheta < \xi/2$.\\
 Then, $\hat f$ is defined by (\ref{hatg}) of Algorithm \ref{algoboost} on ${\cal O}_2$ as
\beqe
\hat f(\XX)= G_{k_{n}}(\bar f), \quad \textrm{with }~\bar f =\sum_{\substack{u\in S^*\\ |u|\leq d}} \sum_{\bolds{l_u}}\beta_{\bolds{l_u}}^{u,0} \vfu(\XuXu). 
\eeqe
%where
%%\beq
%%\bar\eta(\XX)=\sum_{u\in S}\sum_{\bolds{\bolds{l_u}}} \beta_{\bolds{\bolds{l_u}}}^{u}\vfu(\XuXu)\in H_{u,\text{\tiny n}}^{0,L},
%%\eeq  
%$\hat G_{k_n}(\eta)$ is the approximation of $\eta$ at the step $k_n$ defined by .\\
 If we assume that $\mathbf{(H_{s})}$ and $\mathbf{(H_{\varepsilon,q})}$  are satisfied with $q>4/ \xi$, then there exists a sequence $k_{n}:=C \log n$, with $C< (\xi/2-\vartheta) /( 2 \cdot \log 3)$ such that

\beqe
 \|\hat f -\tilde f \| \cvpro 0, \text{when }~ n \rightarrow +\infty.
\eeqe
\end{theor}

We briefly describe the proof and postpone the technical details to the Appendix section.

\begin{proof}[Sketch of Proof of Theorem \ref{thcvboost}]

Mimicking the scheme of~\citet*{buhlmann06} and~\citet*{champion}, the proof first consists in defining the theoretical residual of Algorithm \ref{algoboost} at step $k$ as

\beq\label{theorres}
\barr{lll}
R_k(\bar f)&=&\bar f- G_k(\bar f)\\
&=& \bar f- G_{k-1}(\bar f)-\gamma \psempnn { Y- G_{k-1}(\bar f)}{\vfs} \cdot \vfs
%&=& R_{k-1}(\bar f)-\gamma \psempnn { \tilde f -\bar f+\bar f+\varepsilon- G_{k-1}(\bar f)}{\vfs} \cdot \vfs\\
%&=&  R_{k-1}(\bar f)- \gamma \psempnn{R_{k-1}(\bar f)}{\vfs} \vfs - \gamma \psempnn { \varepsilon}{\vfs} \cdot \vfs \\
%&&- \gamma \psempnn { \tilde f -\bar f}{\vfs} \cdot \vfs.
\earr
\eeq

Further, following the work of~\citet*{champion}, we introduce a \emph{phantom} residual in order to reproduce the behaviour of a deterministic boosting, studied in~\citet*{temlyakov00}. This \emph{phantom} algorithm is the theoretical $\mathbb L_2$-boosting, performed using the randomly chosen elements of the dictionary by Equations \eqref{eqphi2} and \eqref{hatg}, but updated using the deterministic inner product.
The \emph{phantom} residuals $\tilde R_k(\bar f)$, $k\geq 0$, are defined as follows,

\beq\label{phantom}
\left\{
\barr{l}
\tilde R_0(\bar f)=\bar f\\
\tilde R_k(\bar f)=\tilde R_{k-1}(\bar f)-\gamma \psth{\tilde R_{k-1}(\bar f)}{\vfs} \vfs,
\earr
\right.
\eeq
where $\vfs$ has been selected with Equation (\ref{eqphi2}) of Algorithm \ref{algoboost}. The aim is to decompose the quantity $\norm{\hat f-\tilde f}$ to introduce the theoretical residuals and the \emph{phantom} ones,

%\beq\label{decompo}
%\barr{lll}
%\norm{\hat f-\tilde f}&=&\norm{G_{k_{n}}(\bar f)-\tilde f}\\
%&\leq & \norm{G_{k_{n}}(\bar f)-\bar f}+\norm{\bar f-\tilde f}\\
%&\leq & \norm{\bar f-\tilde f} + \norm{R_{k_{n}}(\bar f)-\tilde  R_{k_{n}}(\bar f)}+ \norm{\tilde  R_{k_{n}}(\bar f)}.
%\earr
%\eeq

\beq\label{decompores}
\norm{\hat f-\tilde f}=\norm{G_{k_{n}}(\bar f)-\tilde f} \leq \norm{\bar f-\tilde f} + \norm{R_{k_{n}}(\bar f)-\tilde  R_{k_{n}}(\bar f)}+ \norm{\tilde  R_{k_{n}}(\bar f)}.
\eeq
% and, by (\ref{theorres})-(\ref{phantom}), we get
%
%\beq\label{decompores}
%\barr{lll}
% \norm{R_{k_n}(\bar f)-\tilde  R_{k_n}(\bar f)}&\leq &\norm{ R_{k-1}(\bar f)  - \tilde R_{k-1}(\bar f)}  \\
%& + & \norm{\gamma( \psempnn{R_{k-1}(\bar f)}{\vfs} \vfs-\psth{\tilde R_{k-1}(\bar f)}{\vfs} \vfs )} \\
%& + & \norm{\gamma  \psempnn { \varepsilon}{\vfs} \cdot \vfs }+   \norm{ \gamma \psempnn { \tilde f -\bar f}{\vfs} \cdot \vfs }
%\earr
%\eeq
We then have to show that each term of the right-hand side of (\ref{decompores})  converges towards zero in probability. 
\end{proof}

\par

\setcounter{chapter}{4}
\setcounter{equation}{0} %-1
\noindent {\bf 4. Numerical Applications}\label{sec:num}

In this section, we are interested by the numerical  efficiency of the Two-Steps procedure given in Section 2, and we primarily focus on the practical use of the HOFD through sensitivity analysis (SA). The goal of SA is to identify and to rank the input variables that drive the uncertainty of the model output. For further details, the reader may refer to~\citet*{saltelli,cacuci}. Therefore, the HOFD presented in Paragraph 2.2 is of great interest, because it may be used to decompose the global variance of the model. Here, as each HOFD is subject to hierarchical orthogonality constraints given in Definition \ref{def1}, we obtain that
\beqe
V(Y)= \sum_{u \in S^*} \left [ V(f_u(\XuXu)) + \sum_{u \cap v \neq u,v} \cov(f_u(\XuXu), f_v(\XvXv)) \right]
\eeqe
 Therefore, to measure the contribution of $\XuXu$, for $|u|\geq 1$, in terms of variability in the model, it is then quite natural to define a sensitivity index $S_u$ as follows,
 \beqe
 S_u=\frac { V(f_u(\XuXu)) + \sum_{u \cap v \neq u,v} \cov(f_u(\XuXu), f_v(\XvXv))}{V(Y)}.
 \eeqe
This definition is given and discussed in~\citet*{chastaing}. In practice, once we have applied the procedure described in Algorithm \ref{GHOFD} to get $(\hat f_u, \hat f_v, u \cap v\neq u,v)$, it is straightforward to deduce the empirical estimation of $S_u$, for all $u$. In the following, we are mostly interested by the estimation of the first and second order sensitivity indices (i.e. $S_i$ and $S_{ij}$, $i, j \in \disc 1 p$).\\

%we estimate the interaction components $f_u$ by the procedure described in Algorithm \ref{GHOFD}, and we deduce the estimation of $S_u$ by replacing the second-order moments by their empirical version. }

%The goal of SA is to identify and to rank the input variables that drive the uncertainty of the model output~\citet*{saltelli,saltelliglobal,cacuci}. Thus, the representation of any squared integrable model function as a ANOVA decomposition is very attractive SA because it leads to the variance decomposition, where each partial variance quantifies the variability contribution of a (group of) variable(s) in the model~\citet*{sobol,chastaing,chastaing2}. Another objective of this paper is to apply the $\mathbb L_2$-boosting to sensitivity analysis, and to numerically  compare its performance with the LARS technique~\citet*{efron}, performed for the Lasso regression~\citet*{tibshirani} for the sensitivity analysis purpose.\\
{\bf 4.1 Description}

We end the work with a short simulation study and we are primarily interested by the performance of the greedy selection algorithm for the prediction of generalized sensitivity indices.
As the estimation of these indices consists in estimating the summands of the generalized functional ANOVA decomposition (called HOFD), we start by constructing a hierarchically orthogonal system of functions to approximate the components.
As pointed above (see Assumption $\mathbf{(H_{b}^{3,\vartheta})}$ in Theorem \ref{theo:base} and \ref{thcvboost}), the invertibility of each linear system plays an important role in our theoretical study. We hence have measured for each situation the degeneracy of involved matrices given by

\[
d(A)=\inf_{i,j\in [1:p]} \textrm{det}( A^{ij}).
\]

 Then, we use a variable selection method to select a sparse number of predictors. The goal is to numerically compare three variable selection methods:  the $\mathbb L_2$-boosting, the Forward-Backward greedy algorithm (refered as FoBa in the sequel), and the Lasso estimator.  As pointed above, we have in hand a $n$-sample of i.i.d. observations $(y^s,\xx^s)_{s=1,\cdots,n}$ broken up into two samples of size $n_1=n_2=n/2$. The first sample is used to  construct the system of functions according to Algorithm \ref{EHOFD}. Let us now briefly describe how we use the Lasso and the FoBa.
Each of the three selection methods aims to solve a generic minimization problem
\beqe
(\hat\beta_{\lu}^u)_{\lu,u} \in \argmin_{\beta_{\lu}^ u \in \mathbb R} \frac 1 {n_2} \sum_{s=1}^{n_2} \bigg[  y^s- \sum_{\substack{u\in S\\ |u|\leq d}} \sum_{\lu}\beta_{\lu}^u \vfu(\xuxu^s)  \bigg]^2+ \lambda J(\beta_1^1,\cdots, \beta_{\lu}^ u,\cdots),
\eeqe

{\bf 4.2 Feature selection Algorithms}

\paragraph{FoBa procedure}

The FoBa algorithm, as well as the $\mathbb L_2$-boosting, uses a greedy exploration to minimize the previous criterion when $J(\cdot)$ is a $\ell_0$ penalty, i.e.
$$ J(\beta_1^1,\cdots, \beta_{\lu}^ u,\cdots)=\sum_{\substack{u\in S^*\\ |u|\leq d}} \sum_{\lu} \mathds 1(\beta_{\lu}^ u\neq 0).
$$ 
This algorithm is an iterative scheme that sequentially selects or deletes an element of $\mathcal D$ that has the least impact on the fit, i.e. that significantly reduces the model residual.  This algorithm is described in~\citet*{zhang}, and exploited for HOFD in~\citet*{chastaing2}. We refer to these references for a deeper description of this algorithm. This procedure depends on two shrinkage parameters $\epsilon $ and $\delta$. The parameter $\epsilon$ is the stopping criterion, that predefines if a large number of predictors is going to be introduced in the model. The second parameter, $\delta \in ]0,1]$ offers a flexibility in the \emph{backward} step, as it allows the algorithm to smoothly eliminate at each step a predictor.

In our numerical experiments, we have found a well suited behaviour of the FoBa procedure with $\epsilon=10^{-2}$ and $\delta=1/2$.

\paragraph{Calibration of the Boosting}

We have set $\gamma=0.7$ since it has been previously reported in \citet*{champion} that it was a suitable value for high dimensional regression. 
As we do not know a priori the optimal value for $k_{\textrm{up}}$, we use a $C_p$-Mallows type criterion to fix the optimal number of  iterations. We follow the recommendations of~\citet*{efron} to select the best solution in the LARS algorithm. First, we define a large number of iterations, say $K$. For each step $k \in \{1,\cdots,K\}$, the boosting algorithm computes an estimation of the solution $\hat{\bolds \beta}(k)$. From this, we compute the following quantity,

\beqe
E_k^{\mathrm{Boost}}=\frac 1 n \sum_{s=1}^{n_2} \bigg[ y^s -\sum_{\vfu \in \mathcal D} \hat\beta_{\lu}^u(k) \vfu(\xuxu^s)  \bigg]^2 -n_2+ 2k,
\eeqe 
 where the implied set of functions $\vfu$ have been selected through the first $k$ steps of the algorithm. At last, we choose the optimal number of selected functions $\hat k_{\textrm{up}}$ such that
\beqe
\hat k_{\textrm{up}}=\argmin_{k=1,\cdots,K}E_k^{\mathrm{Boost}}.
\eeqe  

%Je n'ai pas dit qu'on prenait K=30...

\paragraph{Lasso algorithm}

  As the $\ell_0$ strategy is very difficult to handle and may suffer from a lack of robustness,  the $\ell_0$ penalty is often replaced by the $\lambda \times \ell_1$ one, that yield to the Lasso estimator for a given penalization parameter $\lambda>0$. A numerical way to solve it is to use the LARS regression, described in~\citet*{efron} and we refer to this standard reference for a sharp description of this procedure.

Admitting that for a given $\lambda>0$, the Lasso regression admits a unique solution, as described in~\citet*{tibshirani},~\citet*{efron} show that the estimated solution with LARS coincide with the theoretical regularization path $\hat{\bolds\beta}(\lambda)$. The LARS algorithm performs the Lasso regression by offering a set of solutions $\{\hat{\bolds\beta}(\lambda), ~\lambda \in \mathbb R^+\}$. However, the "best" $\lambda$ must be determined to only obtain one solution. In this view, we consider here the criterion defined in~\citet*{efron}. At each step $k$ of the algorithm, the following quantity is computed,

\beqe
E_k^{\mathrm{Lars}}= \left\| \mathbb Y-\mathbb X \hat{\bolds\beta}(\lambda_k)\right\|_{n_2}^2  -n_2 +2k
\eeqe

where $\lambda_k$ is the regularization parameter of the $k$th step. The optimal  $\hat \lambda=\lambda(\hat k)$ is selected such that $\hat k=\argmin_{k} E_k^{\mathrm{Lars}}$ and we keep for the Lasso estimator  $\hat \beta(\hat\lambda)$.\\

{\bf 4.3 Data sets}

Each experiment on each data set has been randomly reproduced 50 times to compute the Monte-Carlo errors.

\paragraph{First Data set: the Ishigami function}

Well known in sensitivity analysis, the analytical form of the Ishigami model is given by,
\beqe
Y=\sin(X_1)+a \sin^2(X_2)+b X_3^4 \sin (X_1),
\eeqe
 where we set $a=7$ and $b=0.1$, and where it is assumed that the inputs are independent. In the numerical experience, we consider the following cases. 
 
 \begin{enumerate}
%\item \label{step1} The inputs are normally distributed as
% 
% \beqe
% \XX\sim N(0,\Sigma), \quad \Sigma=\begin{pmatrix}
% 4 & 0& 0\\ 
% 0&1&0\\
% 0&0&1
% \end{pmatrix}, \quad \textrm{and}
% \eeqe
%\begin {enumerate}
%
%\item \label{step1a} we take $n=200$ observations repeated $50$ times. For the construction of $(\vfu)$, we choose a Hermite basis of degree $L=10$;
%\item \label{step1b} we consider $n=100$ observations repeated $50$ times. For the construction of $(\vfu)$, we choose a Legendre basis of degree $L=15$.
%\end{enumerate}
\item   \label{step1}  For all $i=1,2,3$, the inputs are uniformly distributed  on $[-\pi,\pi]$. We choose $n=300$ observations,
with the first $8$ Legendre basis functions ($L=8$).
 \item \label{step2} 
 For all $i=1,2,3$, the inputs are uniformly distributed on $[-\pi,\pi]$. We choose $n=300$ observations,
  with the first $8$ Fourier basis functions.
\end{enumerate}

Each time, the number of predictors is $m_n= p L + \binom p 2 L^2 = 408 \geq n$.

%For every case, we compare the estimated sensitivity indices to their analytical values in Figure \ref{figishiGBL} and \ref{figishiGBL2}.  Table \ref{tabishig} brings information about the number of retained functions, and the computational cost for each strategy. 

\paragraph{Second Data set: the $g$-Sobol function}

This function is referred in \citet*{saltelli}, and is given by
\[
Y=\prod_{i=1}^p \frac{\abs{4X_i-2}+a_i}{1+a_i}, \quad a_i\geq 0,
\]
where the inputs $X_i$ are independent and uniformly distributed over $[0,1]$. The analytical Sobol indices are given by

\[
S_u=\frac{1}{D}\prod_{i\in u}D_i, \quad D_i=\frac{1}{3(1+a_i)^2}, ~~ D=\prod_{i=1}^p (D_i+1)-1, ~\forall~u \subseteq [1: p].
\]
 Here, we give $a=( 0,
    1,
    4.5,
    9,
   99,
   99,
   99,
   99,
   99,
   99)$.
 For the construction of the hierarchical basis functions, we choose the first $5$ Legendre polynomials ($L=5$). The ANOVA  representation is approximated by first and second order interaction effects, i.e. $d=2$. We use $n=700$ evaluations of the model and the number of predictors $m_n =p L+\binom p 2 L^2=1175$, which clearly exceeds the sample size $n$.\\

{\bf 4.4 The tank pressure model}

This real case study concerns a shell closed by a cap and subject to an internal pressure. Figure \ref{shell} illustrates a simulation of tank distortion. We are interested in the von Mises stress, detailed in~\citet*{vonmises} on the point $y$ labelled in Figure \ref{shell}. The von Mises stress allows for predicting material yielding which occurs when it reaches the material yield strength. 
The selected point $y$ corresponds to the point for which the von Mises stress is maximal in the tank. Therefore, we want to prevent the tank from material damage induced by plastic deformations.
To offer a large panel of tanks able to resist to the internal pressure, a manufacturer wants to know the most contributive parameters to the von Mises criterion variability. 
In the model we propose, the von Mises criterion depends on three geometrical parameters: the shell internal radius ($R_{int}$), the shell thickness ($T_{shell}$), and the cap thickness ($T_{cap}$). 
It also depends on five physical parameters concerning the Young's modulus ($E_{shell}$ and $E_{cap}$) and the yield strength ($\sigma_{y,shell}$ and $\sigma_{y,cap}$) of the shell and the cap. 
The last parameter is the internal pressure ($P_{int}$) applied to the shell. 
The system is modelized by a 2D finite elements code ASTER. In table \ref{tabshell}, we give the input distributions.

\begin{figure}[H]
\centering
 \includegraphics[width=0.4\textwidth]{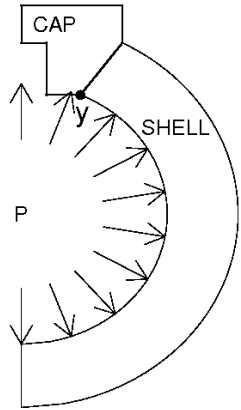}
\caption{Tank distortion at point $y$}\label{shell}
\end{figure} 

\begin{table}
\centering
\renewcommand{\arraystretch}{1.5}
\begin{tabular}{|c|c|}
\hline
Inputs  & Distribution\\
\hline
$R_{int}$ &  $\mathcal U([1800;2200])$, $\gamma(R_{int},T_{shell})=0.85$\\
$T_{shell}$  & $\mathcal U([360;440])$, $\gamma(T_{shell},T_{cap})=0.3$\\
$T_{cap}$  &  $\mathcal U([180;220])$, $\gamma(T_{cap}, R_{int})=0.3$\\
\hline
$E_{cap}$ & $\alpha N(\mu,\Sigma)+(1-\alpha)N(\mu,\Omega)$\\
$\sigma_{y,cap}$ & $\alpha=0.02$, $\mu=
\begin{pmatrix}
 210\\
500
\end{pmatrix}
$, $\Sigma=
\begin{pmatrix}
 350 & 0\\
0 &29
\end{pmatrix}
$, $\Omega=
\begin{pmatrix}
175 & 81 \\
81 & 417
\end{pmatrix}$   \\
\hline
$E_{shell}$  & $\alpha N(\mu,\Sigma)+(1-\alpha)N(\mu, \Omega)$\\ 
$\sigma_{y,shell}$ &  $\alpha=0.02$, $\mu=
\begin{pmatrix}
 70\\
300
\end{pmatrix}
$, $\Sigma=
\begin{pmatrix}
 117 & 0\\
0 &500
\end{pmatrix}
$, $\Omega=
\begin{pmatrix}
58 & 37 \\
37 & 250
\end{pmatrix}$ \\
\hline
$P_{int}$  & $N(80,10)$\\
\hline
\end{tabular}
\caption{Description of inputs of the shell model}\label{tabshell}
\end{table}

 The geometrical parameters are uniformly distributed because of the large choice left for the tank building. The correlation $\gamma$ between the geometrical parameters is induced by the constraints of manufacturing processes. The physical inputs are normally distributed and their uncertainty are due to the manufacturing process and the properties of the elementary constituents variabilities. The large variability of $P_{int}$ in the model corresponds to the different internal pressure values which could be applied to the shell by the user. \\
To measure the contribution of the correlated inputs to the output variability, we estimate the generalized sensitivity indices. 
We proceed to $n=1000$ simulations. 
We use the first Hermite basis functions whose maximum degree is $5$ for  every parameters. \\
%The first order indices dispersions are displayed in Figure \ref{fig_vessel} for Boosting, FoBa and LARS algorithm.

{\bf 4.5 Results}

We consider both the estimation of the sensitivity indices, the ability to select the good representation of the different signals, and the computation time needed to obtain the sparse representation. "Greedy" refers to the Foba procedure as well as "LARS" refers to the Lasso resolution, and we refer to our method as "Boosting". 

\paragraph{Sensitivity estimation}
Figures \ref{figishiGBL} and \ref{figishiGBL2} provide the dispersion of the sensitivity indices estimated by our three methods on the Ishigami function. We can see that the three methods behave well with the two basis. Note that handling the Fourier basis is, as expected, more suitable for the Ishigami function than the Legendre basis (see the sensitivity index $S_3$ in Figures \ref{figishiGBL} and \ref{figishiGBL2}). We can also draw similar conclusions with Figure \ref{fig_gsobol}, where the three methods yields the same conclusion. Note also that the standard deviations of each method seem quite equivalent. 

At last, as pointed by Figure \ref{fig_vessel}, the most contributive parameter to the von Mises criterion variability is the internal pressure $P_{int}$, which is not surprising. Concerning now the geometric characteristics, the three methods exhibit as main parameters the cap thickness $T_{cap}$ and the shell thickness $T_{shell}$ using their expensive code although the shell internal radius does not seem so important.

  \begin{figure}
\centering
 \includegraphics[width=0.6\textwidth, angle=270]
{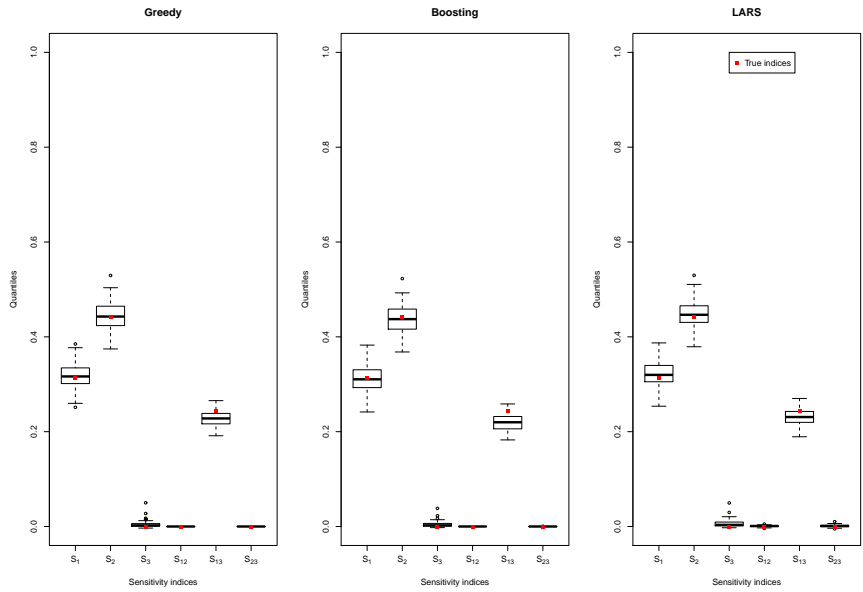}
\caption{Representation of the first-order components on the First Data set (Ishigami function) described through the Fourier basis}\label{figishiGBL}
\end{figure}

  \begin{figure}
\centering
 \includegraphics[width=0.6\textwidth, angle=270]
{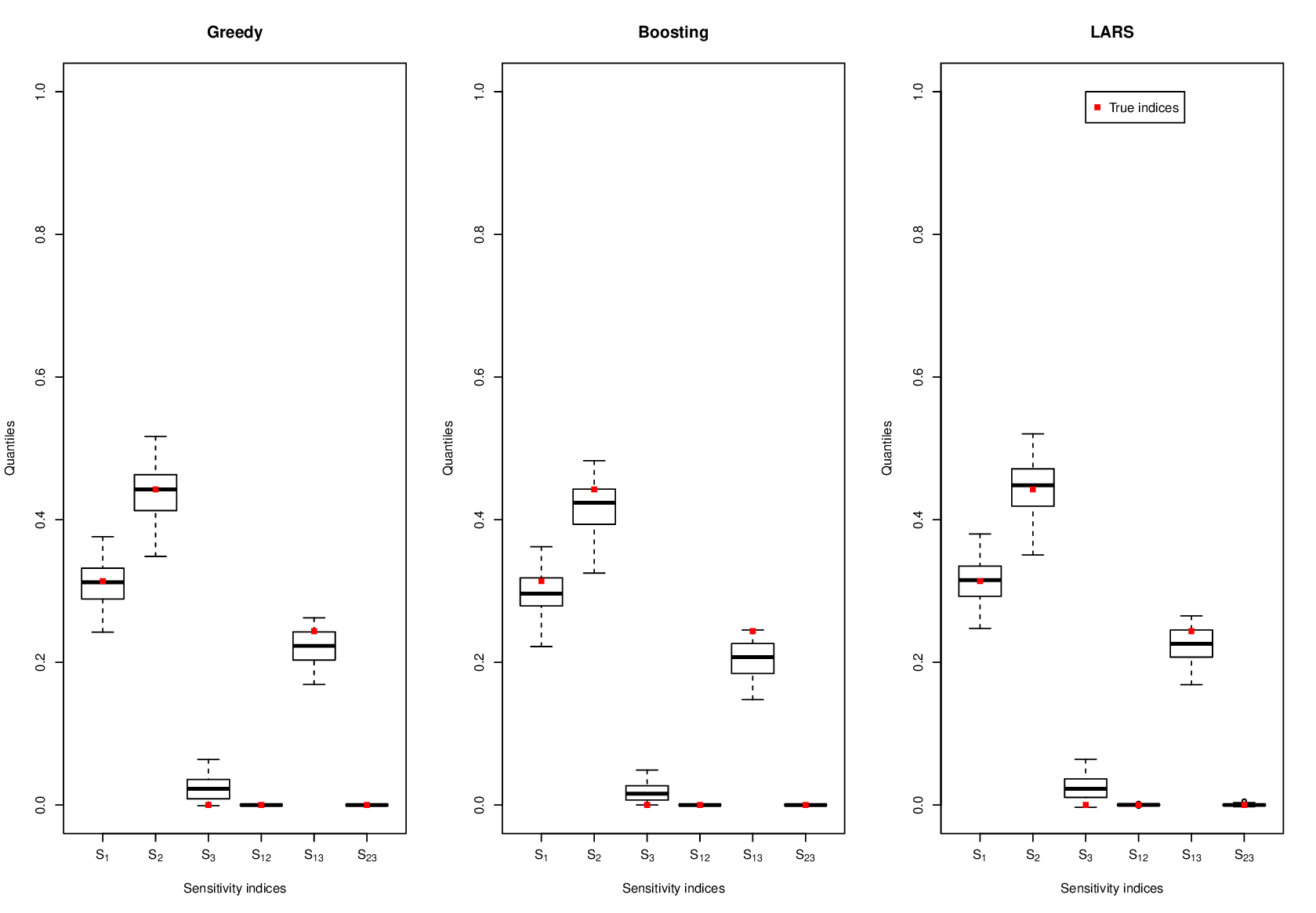}
\caption{Representation of the first-order components on the First Data set (Ishigami function) described through the  Legendre basis}\label{figishiGBL2}
\end{figure}

   \begin{figure}
\centering
 \includegraphics[width=0.6\textwidth, angle=270]
{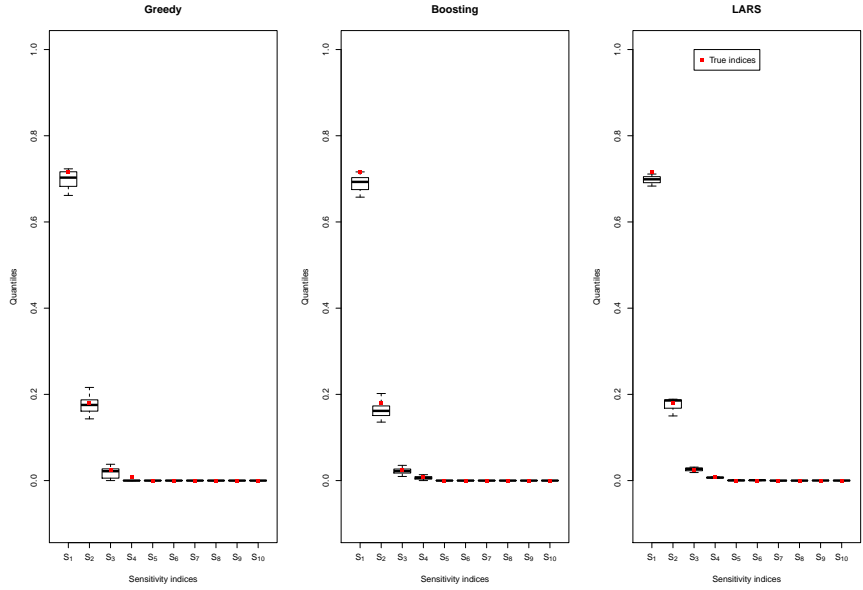}
\caption{Representation of the first-order components on the Second Data set  ($g$-Sobol function)}\label{fig_gsobol}
\end{figure}

% \begin{figure}[H]
%\centering
% \includegraphics[width=0.7\textwidth, angle=270]{./Graphes/limodel_GBL.eps}
%\caption{Estimation of the sensitivity indices on the Third dataset  (Polynomial function)}\label{figliGBL}
%\end{figure}

  \begin{figure}
\centering
 \includegraphics[width=0.6\textwidth, angle=270]
{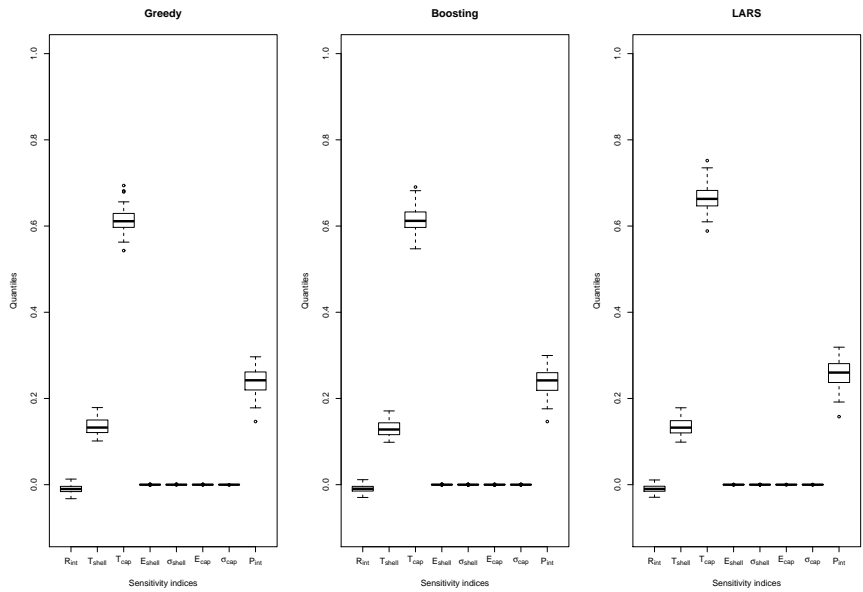}
\caption{Dispersion of the first order sensitivity indices of the tank model parameters}\label{fig_vessel}
\end{figure}

\paragraph{Computation time and accuracy}
We enumerate in Table \ref{tab:perf} the performances of the three methods, according to their computational cost, and accuracy of the feature selection.

 \begin{table}[H]
\centering
\begin{tabular}{|c|c|c|c|}
 \hline
 Data set & Procedure &   $\norm{\hat{\bolds{\beta}}}_0$ & Elapsed Time (in sec.) \\
\hline
\multirow{3}{1.7cm}{Ishigami function Case \ref{step1}} & $\mathbb L_2$-boosting &19   &   $0.0941$  \\
&FoBa &  21 & 2.2917 \\
& LARS &  50 &   53.03    \\
\hline
\multirow{3}{1.7cm}{Ishigami function Case \ref{step2}} & $\mathbb L_2$-boosting & 15   &   $0.0884$   \\
&FoBa &  12 & 1.0752 \\
& LARS &  45 &   23.2062   \\
\hline
\multirow{3}{1.7cm}{$g$-Sobol function } & $\mathbb L_2$-boosting & 7.4   &   $1.0620$ \\
&FoBa &  4.7 & 2.9195  \\
& LARS &   &   $10^3$  \\
\hline
%\multirow{3}{1.7cm}{Polynomial function } & $\mathbb L_2$-boosting &8& 6.5   &  $0.1972$ \\
%&FoBa &8 &7 & $0.6318$ \\
%& LARS & 8& 110 &   $57.658$   \\
%\hline
\multirow{3}{1.8cm}{Tank pressure model } & $\mathbb L_2$-boosting & 10   &  $0.0266$  \\
&FoBa & 22 & $0.3741$ \\
& LARS &  10 &   $0.1756$   \\
\hline
\end{tabular}
\caption{Features of the three algorithms}\label{tab:perf}
\end{table}

It clearly appears in Table \ref{tab:perf} that our proposed $\mathbb L_2$-boosting is the fastest method. Also, although we do not have access to the theoretical support recovery $\norm{\bolds{\beta}}_0$, we notice that the $\mathbb L_2$-boosting selects a small number of predictors, and yet performs quite well through the applications. This presumes that the $\mathbb L_2$-boosting is more accurate, as it seems to make a good support recovery.
% and almost the more accurate one regarding the support recovery performances (size of the estimated support $\norm{\hat{\bolds{\beta}}}_0$ compared to the size of the theoretical one $\norm{\bolds{\beta}}_0$). Indeed, 
 The FoBa procedure performances are also very good regarding their ability to obtain a sparse representation and the fraction of additional time required by this last algorithm in comparison with the $\mathbb L_2$-boosting oscillates between two and about ten, or so. At last, the LARS algorithm possesses a somewhat larger computational cost although its performances on our several data sets were quite disappointing.

Note that we have computed the maximal "degeneracy" which is involved in the resolution of the linear systems and quantified by Assumption $\mathbf{(H_{b}^{3,\vartheta})}$ in the column 2 of Table \ref{tab:deg}. In many cases, we obtain a significantly larger value than $0$. The third column of Table \ref{tab:deg} shows the admissible size of the parameter $\vartheta$ and
 we can check that the number of variables $p_n$  allowed by $\mathbf{(H_{b}^{2})}$ and the balance between $\xi$ and $ \vartheta$
 ($\xi$ should be greater than $2 \vartheta$ in our theoretical results)  is not restrictive since $n^{1-2 \vartheta}$ is always significantly greater than $log(m_n)$ in Table \ref{tab:deg}.

 \begin{table}[H]
\centering
\begin{tabular}{|c|c|c|c|c|}
\hline
Data set & Degeneracy $d(A)$& $\vartheta \geq  \frac{\log(1/d(A))}{log(n)}$ & $n^{1-2 \vartheta}$ & $\log(m_n)$ \\
\hline
Ishigami function Case\ref{step1} & $0.6388$ & $ [0.0786,+\infty[$ & $191.6106$ & $6.0113$ \\
Ishigami function Case\ref{step1} & $0.76$ & $ [0.0481,+\infty[$ & $228.0194$ & $6.0113$ \\
$g$-Sobol function & $0.9410$ & $[0.0093,+\infty[$ & $619.6967$ &  $7.0690$ \\
Polynomial function &$0.2736$& $[0.2446,+ \infty[$& $14.9750$ & $5.7991$\\
\hline
\end{tabular}
\caption{Degeneracy of the linear systems and admissible size of $p_n$}\label{tab:deg}
\end{table}

\par

\setcounter{chapter}{1}
{\bf 5. Conclusions and Perspectives}\label{sec:ccl}

This paper brings a rigorous framework for the hierarchically orthogonal Gram-Schmidt procedure in a high-dimensional paradigm, when the greedy $\mathbb L_2$-boosting is used. It also appears that we obtain satisfying numerical results through our three Data sets with a very low computational cost. From a mathematical point of view, assumption $(\mathbf{H_b^1})$ presents a restrictive condition, and to relax it would open a wider class of basis functions for applications. We let this development open for a future work, which may rely either on a development of a concentration inequality for unbounded random matrices or on a truncating argument.\\

\newpage

%\section{Notation and reminder}\label{sec:notation}

\setcounter{chapter}{6}
\setcounter{equation}{0} %-1
\noindent {\bf 6 Appendix}

\noindent {\bf 6.1 Notation and reminder}\label{sec:notation}
%{\bf 2.1 Notation}\label{sec:not}
%\noindent {\bf 2. Estimation of the generalized Hoeffding decomposition components}\label{sec:notation}
%\section{Estimation of the generalized Hoeffding decomposition components}\label{sec:notation}

%{\bf 2.1 Notation}\label{sec:not}

Let us first recall some standard notation on matricial norms. For any square matrix $M$, its spectral radius $\rho(M)$ will refer to the largest absolute value of the elements of its spectrum:
$$
\rho(M) := \max_{ \alpha \in Sp(M)} |\alpha |.
$$
Moreover, $\normmat{M}$ is the euclidean endomorphism norm  and is given by
$$
\normmat{M} := \sqrt{\rho(M^t M)},
$$
where $M^t$ is the transpose of $M$. Note that for self-adjoint matrices, $\normmat{M}=\rho(M)$.
At last, the Frobenius norm of $M$ is given by
$$
\|M\|_F := \left( Tr(M^t M) \right)^{1/2}.
$$

\noindent {\bf 6.2 Hoeffding 's type Inequality for random bounded matrices}\label{sec:tropp}
%\section{Hoeffding 's type Inequality for random bounded matrices}\label{sec:tropp}
%\setcounter{equation}{0}
For sake of completeness, we quote here  Theorem 1.3 of~\citet*{tropp11}.
\begin{theor}[Matrix Hoeffding: bounded case]
Consider a finite sequence $(X_k)_{1 \leq k \leq n}$ of independent random self-adjoint matrices with dimension $d$, and let $(A_k)_{1\leq k \leq n}$ a deterministic sequence of self-adjoint matrices. Assume that
$$
\forall 1 \leq k \leq n \qquad \mathbb{E} X_k = 0 \qquad \text{and} \qquad X_k^2 \preceq A_k^2 \quad a.s.
$$
Then, for all $t \geq 0$
$$
P \left(\lambda_{max}  \left(\sum_{k=1}^n X_k \right) \geq t \right) \leq d e^{-t^2/8 \sigma^2}, \qquad \text{where} \qquad \sigma^2 = \|\sum_{k=1}^n A_k^2 \|.
$$
\end{theor}
In our work, it is useless to use a more precise concentration inequality such as the Bernstein one (see Theorem 6.1 of~\citet*{tropp11}) since we do not consider any asymptotic on $L$ (the number of basis functions for each variables $X^j$). Such asymptotic setting is far beyond the scope of the paper and we let this problem open for a future work.

\noindent {\bf 6.3 Proof of Theorem 1\label{sec:theo1}}
%\section{Proof of Theorem 1\label{sec:theo1}}

%\setcounter{equation}{0}

Consider any subset  $u=(u_1,...,u_t) \in S^*$ with $t\geq 1$ and remark that if  $u = \{i\}$, i.e. $t=1$, and $L \geq 1$, we have seen in the \emph{Initialization} of Algorithm 1 that
 \beqe
\vfi=\pfi, \quad \forall~l_i \in \disc 1 {L},
\eeqe
 Therefore, we obviously have that $\sup_{\substack{i \in \disc 1 p\\ l_i\in \disc 1 {L}}} \norm{\vfi-\pfi}=0$.
%In the sequel, if $u = \{i\}$, i.e. $t=1$, and $L \geq 1$, we assume that 
 %\beqe
%\varphi_{l_i}^i=\phi_{l_i}^i, \quad \forall~l_i \in \disc 1 {L},
%\eeqe
%and 
%$L_i=L$ for all $i=1,\cdots,p$.

Now, for $t=2$, let $u=\{i,j\}$, with $i \neq j \in \disc 1 p$, and $\bolds{\bolds{l_{ij}}}=(l_i,l_j) \in \disc 1 L ^2$, remind that $\pfij$ is defined as:

\beqe
\pfij(\xixi,\xjxj)=\pfi(\xixi)\times \pfj(\xjxj) + \sum_{k=1}^L \lambda_{k,\bolds{l_{ij}}}^i \phi_k^i (\xixi) + \sum_{k=1}^L \lambda_{k,\bolds{l_{ij}}}^j \phi_k^j (\xjxj)+ C_{\bolds{\bolds{l_{ij}}}},
\eeqe
where $(C_{\bolds{\bolds{l_{ij}}}},(\lambda_{k,\bolds{l_{ij}}}^i)_k, (\lambda_{k,\bolds{l_{ij}}}^j)_k)$ are given as the solutions of:
\beq\label{syst1}
\begin{array}{l}
\psth{\pfij}{\phi_k^i}=0, \quad \forall~k \in \disc 1 L\\
\psth{\pfij}{\phi_k^j}=0, \quad \forall~k \in \disc 1 L\\
\psth{\pfij}{1}=0.
\end{array}
\eeq

When removing $C_{\bolds{\bolds{l_{ij}}}}$, the resolution of (\ref{syst1}) leads to the resolution of a linear system of the type:
 \beq\label{formemat1}
 A^{ij} \bolds{\lambda}^{\bolds{l_{ij}}}= D^{\bolds{l_{ij}}},
 \eeq
  with $\bolds\lambda^{\bolds{l_{ij}}}=\left(  \lambda_{1,\bolds{l_{ij}}}^i \cdots  \lambda_{L,\bolds{l_{ij}}}^i  \lambda_{1,\bolds{l_{ij}}}^j  \cdots  \lambda_{L,\bolds{l_{ij}}}^j \right)^t$ and 
 \beqe
 A^{ij}=\begin{pmatrix}
 B^{ii} &  B^{ij} \\
 {}^t  B^{ij} &  B^{jj}
 \end{pmatrix}, \quad
 B^{ij}=\begin{pmatrix}
 \psth{\phi_1^i}{\phi_1^j} & \cdots & \psth{\phi_1^i}{\phi_L^j}\\
 \vdots \\
 \psth{\phi_L^i}{\phi_1^j} & \cdots & \psth{\phi_L^i}{\phi_L^j}\\
 \end{pmatrix},\quad 
 D^{\bolds{l_{ij}}}=-\begin{pmatrix}
 \psth{\pfi\times \pfj}{\phi_1^i}\\
 \vdots\\
  \psth{\pfi\times \pfj}{\phi_L^i}\\
  \psth{\pfi\times \pfj}{\phi_1^j}\\
 \vdots\\
  \psth{\pfi\times \pfj}{\phi_L^j}\\
 \end{pmatrix}.
 \eeqe

 Consider now $\vfij$ which is decomposed on the dictionary as follows:
 \beqe
 \barr{lll} 
\vfij(\xixi,\xjxj)
%&\vfi(\xixi)\times \vfj(\xjxj) + \sum_{k=1}^L \hat\lambda_{k,n}^i \varphi_k^i (\xixi) + \sum_{k=1}^L \hat\lambda_{k,n}^j \varphi_k^j (\xjxj)+ \hat C_{\bolds{\bolds{l_{ij}}}}^n \\
&=& \pfi(\xixi)\times \pfj(\xjxj) + \sum_{k=1}^L \hat\lambda_{k,\bolds{l_{ij}},n_1}^i \phi_k^i (\xixi) \vspace{1em}  + \sum_{k=1}^L \hat\lambda_{k,\bolds{l_{ij}},n_1}^j \phi_k^j (\xjxj)+ \hat C_{\bolds{\bolds{l_{ij}}}}^{n_1},
 \earr
\eeqe
where $(\hat C_{\bolds{\bolds{l_{ij}}}}^{n_1},(\hat\lambda_{k,\bolds{l_{ij}},n_1}^i)_k, (\hat\lambda_{k,\bolds{l_{ij}},n_1}^j)_k)$ are given as solutions of the following \emph{random} equalities:

\beq\label{syst2}
\begin{array}{l}
\psempn{\vfij}{\phi_k^i}=0, \quad \forall~k \in \disc 1 L\\
\psempn{\vfij}{\phi_k^j}=0, \quad \forall~k \in \disc 1 L\\
\psempn{\vfij}{1}=0.
\end{array}
\eeq

When removing $\hat C_{\bolds{\bolds{l_{ij}}}}^{n_1}$, the resolution of (\ref{syst2}) can also lead to the resolution of a linear system of the type:  
 
 \beq\label{formemat2}
 \hat A_{n_1}^{ij} \hat{\bolds{\lambda}}^{\bolds{l_{ij}}}_{n_1}= \hat D_{n_1}^{\bolds{l_{ij}}},
 \eeq
 where $\hat{ \bolds\lambda}^{\bolds{l_{ij}}}_{n_1}= {} \left( \hat\lambda_{1,\bolds{l_{ij}},n_1}^i \cdots \hat \lambda_{L,\bolds{l_{ij}},n_1}^i \hat\lambda_{1,\bolds{l_{ij}},n_1}^j  \cdots \hat \lambda_{L,\bolds{l_{ij}},n_1}^j \right)^t$ and  $\hat A^{ij}_{n_1}$ (\textit{resp.} $\hat D^{\bolds{l_{ij}}}_{n_1}$) are obtained from $A^{ij}$ (\textit{resp.} $D^{\bolds{l_{ij}}}$) by changing the theoretical inner product by its empirical version.
% \beqe
% \hat A_n=\begin{pmatrix}
% \hat A_n^{ii} &  \hat A_n^{ij} \\
% {}^t  \hat A_n^{ij} &  \hat A_n^{jj}
% \end{pmatrix}, \quad
%\hat A_n^{ij}=\begin{pmatrix}
% \psemp{\phi_1^i}{\phi_1^j} & \cdots & \psemp{\phi_1^i}{ \phi_L^j}\\
% \vdots \\
% \psemp{\phi_L^i}{\phi_1^j} & \cdots & \psemp{\phi_L^i}{ \phi_L^j}\\
% \end{pmatrix},\quad 
%\hat{ \bolds\lambda}^n=\begin{pmatrix}
% \hat\lambda_{1,n}^i\\
% \vdots\\
% \hat \lambda_{L,n}^i\\
%   \hat\lambda_{1,n}^j\\
%   \vdots\\
%   \hat \lambda_{L,n}^j\\
% \end{pmatrix}, 
% \eeqe
 
%and,
 
% \beqe
% \hat D_n=-\begin{pmatrix}
% \psemp{\pfi\times \pfj}{ \phi_1^i}\\
% \vdots\\
%  \psemp{\pfi\times \pfj}{ \phi_L^i}\\
%  \psemp{\pfi\times \pfj,}{\phi_1^j}\\
% \vdots\\
%  \psemp{\pfi\times \pfj}{ \phi_L^j}\\
% \end{pmatrix}.
% \eeqe
\begin{rem}
 Remark that $A^{ij}$ depends on $(i,j)$ as well as $\bolds\lambda^{\bolds{l_{ij}}}$ and $D^{\bolds{l_{ij}}}$ depend on $(i,j)$ and $\bolds{l_{ij}}$, but we will deliberately omit these indexes in the sequel for sake of convenience when no confusion is possible. 
 For instance, when a couple $(i,j)$ is handled, we will frequently use the notation $A, \bolds\lambda , D, C, \lambda_{k}^i, \lambda_{k}^j$ instead of $A^{ij}, \bolds\lambda^{\bolds{l_{ij}}}, D^{\bolds{l_{ij}}}, C_{\bolds{l_{ij}}}, \lambda_{k,\bolds{l_{ij}}}^i$ and $ \lambda_{k,\bolds{l_{ij}}}^j$. 
 This will be also the case for the estimators $\hat{A}_{n_1}, \hat{ \bolds\lambda}_{n_1}, \hat{ D}_{n_1},\hat{C}^{n_1}, \hat{\lambda}_{k,n_1}^i$ and $\hat{\lambda}_{k,n_1}^j$.
\end{rem}

Then, the following useful lemma compares the two matrices $\hat A_{n_1}$ and $A$.
\begin{lem}\label{lem:A}
Under Assumption  $\mathbf{(H_{b})}$, and for any $\xi$ given by $\mathbf{(H_{b}^2)}$, one has
$$ \underset{1\leq i,j \leq p_n}{\sup} \normmat{\hat A_{n_1} - A} = \gopxi.$$
\end{lem}

\begin{proof}
First consider one couple $(i,j)$ and note that $\normmat{\hat{A}_{n_1}-A}=\rho(\hat{A}_{n_1}-A)$, since $\hat{A}_{n_1}-A$ is self-adjoint. To obtain a concentration inequality on the matricial norm  $\normmat{\hat{A}_{n_1}-A }$, we mainly use the results of~\citet*{tropp11}, which give concentration inequalities for the largest eigenvalue of self-adjoint matrices (see section \ref{sec:tropp}).
Denote $\preceq$ the semi-definite order on self-adjoint matrices, which is defined for all self-adjoint matrices $M_1$ and $M_2$ of size $q$ as:
$$M_1\preceq M_2 \ \ \mbox{iff} \ \ \forall u \in \mathbb{R}^q, \ \ {} u^t M_1 u \leq {} u^t M_2 u.$$ 
 
Remark that $\hat A_{n_1}-A$ could be written as follows:
\beqe
\hat A_{n_1}-A=\frac{1}{n_1} \sum_{r=1}^{n_1} \Theta_{r,ij},\quad \Theta_{r,ij}=
\begin{pmatrix}
\Theta_r^{ii} & \Theta_r^{ij}\\
{}^t \Theta_r^{ij} & \Theta_r^{jj}
\end{pmatrix}, ~~\forall~r\in \disc 1 {n_1},
\eeqe
where, for all $k,m \in \disc 1 L$, $(\Theta_r^{i_1 i_2})_{k,m}
=\pf{k}{i_1}(x_{i_1}^r)\pf{m}{i_2}(x_{i_2}^r)-\mathbb E[\pf{k}{i_1}(X_{i_1})\pf{m}{i_{2}}(X_{i_2})]$ with $i_1,i_2 \in \{i,j\}$. Since the observations $(\xx^r)_{r=1,\cdots,n_1}$ are supposed to be independent, $\Theta_{1,ij},\cdots, \Theta_{n_1,ij}$ is a sequence of independent, random, centered, self-adjoint matrices.
Moreover, for all $u\in \mathbb{R}^{2L}$, all $r \in \disc 1 {n_1}$,
\beqe
{} u^t \Theta_{r,ij}^2 u = \eucl{\Theta_{r,ij} u}^2
	\leq  \eucl{u}^2 \Vert \Theta_{r,ij} \Vert_F^2,
\eeqe
where 
\beqe
\barr{lll}
\Vert \Theta_{r,ij} \Vert_F^2 &\leq & (2L)^2  \left(\max_{k,m\in \disc 1 L} |(\Theta_{r,ij})_{k,m}|\right)^2 \\
	&\leq & (2L)^2 \left( \max_{\substack{k,m\in \disc 1 L\\ i_1,i_2\in \{i,j\}}}  |\pf{k}{i_1}(x_{i_1}^r)\pf{m}{i_2}(x_{i_2}^r)-\mathbb E[\pf{k}{i_1}(X_{i_1})\pf{m}{i_2}(X_{i_2})]| \right)^2\\
&\leq & 16L^2 M^4 \quad \textrm{by $\mathbf{(H_{b}^1)}$}.
\earr
\eeqe
We then deduce that each element of the sum satisfies $X_{l,ij}^2 \preceq 16L^2 M^4 \id_{L^2}$, where $\id_{L^2}$ denotes the identity matrix of size $L^2$.

Applying now the Hoeffding's type Inequality stated in Theorem 1.3 of~\citet*{tropp11} to our sequence $\Theta_{1,ij},\cdots, \Theta_{n_1,ij}$, with $\sigma^2=16n_1L^2M^4,$ we then obtain that
\beqe
\forall t \geq 0 \qquad 
P\left( \rho \left( \frac{1}{n_1} \sum_{r=1}^{n_1} \Theta_{r,ij} \right) \geq t\right) \leq 2L e^{-\frac{(n_1t)^2}{8\sigma^2}},
\eeqe

Considering now the whole set of estimators $\hat{A}_{n_1}$, we obtain

\beqe
\forall t \geq 0 \qquad 
P\left(\underset{1\leq i,j \leq p_n}{\sup} \rho \left( \frac{1}{n_1} \sum_{r=1}^{n_1} \Theta_{r,ij} \right) \geq t\right) \leq 2Lp_n^2 e^{-\frac{(n_1t)^2}{8\sigma^2}},
\eeqe

Now, we take $t=\gamma n^{-\xi/2}$, where $\gamma >0$, and $0< \xi \leq 1$ given in  $\mathbf{(H_{b}^2)}$. Then, the following inequality  holds:
\beq \label{eq:concen}
P\left(\underset{1\leq i,j \leq p_n}{\sup} \rho \left( \hat{A}_{n_1} -A \right) \geq \gamma n^{-\xi/2}\right) \leq 2Lp_n^2 e^{-\frac{{n_1}^{1-\xi} \gamma^2}{128 L^2M^4}} .%=2Lp_n^2 e^{-\frac{n^{1-\xi} \gamma^2}{\red{256} L^2M^4}} .
\eeq
Since $n_1=n/2$, and $p_n=\gophyp$ by Assumption $\mathbf{(H_{b}^2)}$, the right-hand side of the previous inequality becomes arbitrarily small for $n$ sufficiently large and $\gamma >0$ large enough. The end of the proof follows using Inequality (\ref{eq:concen}).
\end{proof}

Similarly, we can show that the estimated quantity $\hat{D}_{n_1}$ is not so far from the theoretical $D$ with high probability.
\begin{lem}\label{lem:D}
Under Assumptions $\mathbf{(H_{b})}$, and for any $\xi$ given by $\mathbf{(H_{b}^2)}$, one has
\beqe
\underset{i,j,\lij}{\sup}  \eucl{\hat D_{n_1}-D} = \gopxi.
\eeqe

\end{lem}

\begin{proof} First consider one couple $(i,j)$.
We aim to apply another concentration inequality on $\eucl{\hat D^{\lij}_{n_1}-D^{\lij}}$. Remark that  $\eucl{\hat D_{n_1}-D}$ can be written as:
\beqe
\barr{lll}
\eucl{\hat D_{n_1}-D}&=& \left( \sum_{k=1}^L \left( \psempn{\pfi\times \pfj}{ \phi_k^i}- \psth{\pfi\times \pfj}{\phi_k^i}\right)^2+ \right.\\
&& \left. \sum_{k=1}^L \left( \psempn{\pfi\times \pfj}{ \phi_k^j}- \psth{\pfi\times \pfj}{\phi_k^j}\right)^2 \right)^{1/2}\\
&\leq & \sum_{k=1}^L \left| \frac 1 {n_1}\sum_{r=1}^{n_1} \pfi(\xixi^r) \pfj(\xjxj^r) \phi_k^i(\xixi^r)- \psth{\pfi\times \pfj}{\phi_k^i} \right|+ \\[0.2cm]
&&\sum_{k=1}^L \left| \frac 1 {n_1}\sum_{r=1}^{n_1} \pfi(\xixi^r) \pfj(\xjxj^r) \phi_k^j(\xjxj^r) - \psth{\pfi\times \pfj}{\phi_k^j} \right|.
\earr
\eeqe
Now, Bernstein's Inequality (see~\citet*{Birge98} for instance) implies that, for all $\gamma >0$,
 \beqe
 \barr{lll}
 P\left( n_1^{\xi/2} \eucl{\hat D_{n_1}-D}\geq\gamma\right ) &\leq &  P\left(n_1^{\xi/2}  \sum_{k=1}^L \left| \frac 1 {n_1}\sum_{r=1}^{n_1} \pfi(x_i^r) \pfj(x_j^r) \phi_k^i(x_i^r)- \psth{\pfi\times \pfj}{\phi_k^i} \right| >\gamma/2\right)\\
 & +&  P\left(n_1^{\xi/2}  \sum_{k=1}^L \left| \frac 1 {n_1}\sum_{r=1}^{n_1} \pfi(x_i^r) \pfj(x_j^r)  \phi_k^i(x_i^r)- \psth{\pfi\times \pfj}{ \phi_k^i} \right| >\gamma/2\right)\\
 &\leq & 4L \exp{\left(-\frac 1 8 \frac{\gamma^2 n_1^{1-\xi}}{M^6+M^3 \gamma/6 n_1^{-\xi/2}}\right)},
 \earr
 \eeqe
which gives:
 \beq \label{eq:concen2}
  P\left(\underset{i,j,\lij}{\sup}  \eucl{\hat D_{n_1}-D}\geq\gamma n_1^{-\xi/2} \right ) 
 \leq 4L\times L^2 p_n^2 \exp{\left(- \frac 1 8 \frac{\gamma^2{n_1}^{1-\xi}}{M^{6}+M^3 \gamma/6 {n_1}^{-\xi/2} }\right)}.
 \eeq
Now,  since $n_1=n/2$, Assumption $\mathbf{(H_{b}^2)}$ implies that the right-hand side of Inequality (\ref{eq:concen2}) can also become arbitrarily small for $n$ sufficiently large, which concludes the proof.
\end{proof}

The next lemma then compares the estimated $\incchap$ with $\inc$.

\begin{lem}\label{lem:lambda}
Under Assumptions $\mathbf{(H_{b})}$, we have when $\vartheta < \xi/2$,
$$\underset{i,j,\lij}{\sup} \eucl{\hat{\bolds\lambda}_{n_1} - \bolds\lambda } = \gopxiv.$$
\end{lem}

%\begin{rem}
% Remark that $A^{ij}$ depends on $(i,j)$ as well as $\bolds\lambda^{\bolds{l_{ij}}}$ and $D^{\bolds{l_{ij}}}$ depend on $(i,j)$ and $\bolds{l_{ij}}$, but we will deliberately omit these indexes in the sequel for sake of convenience when no confusion is possible. 
% For instance, when a couple $(i,j)$ is handled, we will frequently use the notation $A, \bolds\lambda , D, C, \lambda_{k}^i, \lambda_{k}^j$ instead of $A^{ij}, \bolds\lambda^{\bolds{l_{ij}}}, D^{\bolds{l_{ij}}}, C_{\bolds{l_{ij}}}, \lambda_{k,\bolds{l_{ij}}}^i$ and $ \lambda_{k,\bolds{l_{ij}}}^j$. 
% This will be also the case for the estimators $\hat{A}_{n_1}, \hat{ \bolds\lambda}_{n_1}, \hat{ D}_{n_1},\hat{ C}_{n_1}, \hat{\lambda}_{k,n_1}^i$ and $\hat{\lambda}_{k,n_1}^j$.
%\end{rem}

\begin{proof} Fix any couple $(i,j)$, $\inc$ and $\incchap$ satisfy Equations (\ref{formemat1}) and (\ref{formemat2}). Hence,
  \beqe
  \begin{array}{lcll}
 & A(\incchap-\inc)-A\incchap &=&-D= \hat D_{n_1} -D -\hat D_{n_1}\\
&  &=& (\hat D_{n_1}-D)-\hat A_{n_1} \incchap\\
 \Leftrightarrow & A(\incchap-\inc)&=&  (\hat D_{n_1}-D) +(A-\hat A_{n_1}) \incchap\\
  \Leftrightarrow & \incchap-\inc &=& A^{-1}[(A-\hat A_{n_1}) \incchap ]+ A^{-1}  (\hat D_{n_1}-D),
  \end{array}
  \eeqe
  since the matrix $A$ is  positive definite.
It follows that
\beqe
\incchap-\inc = A^{-1}(A-\hat A_{n_1}) (\incchap - \inc) + A^{-1}(A-\hat A_{n_1}) \inc + A^{-1}  (\hat D_{n_1}-D),
\eeqe
and 
\beq \label{eq:inv1}
\left(\id  - A^{-1}(A-\hat A_{n_1})\right) (\incchap-\inc) = A^{-1}(A-\hat A_{n_1}) \inc + A^{-1}  (\hat D_{n_1}-D),
\eeq

Remark that $\normmat{\hat A_{n_1} -A} = \gopxi$ by Lemma \ref{lem:A}. Hence, with high probability and for $n$ large enough
$\id - A^{-1}(A-\hat A_{n_1}) $ is invertible, and Inequality (\ref{eq:inv1}) can be rewritten as:
\beqe
\incchap-\inc = \left(\id - A^{-1}(A-\hat A_{n_1})\right)^{-1} \left( A^{-1}(A-\hat A_{n_1}) \inc + A^{-1}  (\hat D_{n_1}-D) \right).
\eeqe
We then deduce that, 

\beq \label{eq:ineg2}
\barr{lll}
\eucl{\hat{\bolds\lambda}_{n_1} - \bolds\lambda} &\leq & \normmat{\left(\id - A^{-1}(A-\hat A_{n_1})\right)^{-1}} \\
& & \times  \left( \normmat { A^{-1}[A-\hat A_{n_1}] } \eucl{\inc} + \eucl{A^{-1}  (\hat D_{n_1}-D)} \right)\\[0.4cm]
	&\leq & \normmat{\left(\id - A^{-1}(A-\hat A_{n_1})\right)^{-1}}\\[0.4cm]
	& & \times  \left( \normmat{A^{-1}} \normmat { A-\hat A_{n_1}} \eucl{\inc} + \normmat{A^{-1}} \eucl{\hat D_{n_1}-D} \right).
	\earr
\eeq
%    where $\normmat{\cdot }$ denotes the matricial norm induced by the euclidian norm, which satisfies:
%  
%\beqe
%\normmat{T}=\sup_{\eucl{x}} \eucl{Tx}=\sqrt{\rho({}^tTT)},
%\eeqe
%
%  
%where $\rho({}^t TT):=max_i\{|\lambda_i|,~\lambda_i \in \lambda({}^t TT)\}$ is the spectral radius of ${}^tTT$. 
%

%  
%  \beqe
%  \normmat{ A-\hat A_n } \leq\normmatf {A-\hat A_n}= \textrm{Trace}[(A-\hat A_n){}^t (A-\hat A_n)]
%  \eeqe
%  
%  Après calcul, la trace peut s'écrire comme suit, 
%  
%  \beqe
%  \begin{array}{lll}
%   \textrm{Trace}[(A-\hat A_n){}^t (A-\hat A_n)] &=& \sum_{k=1}^L \sum_{m=1}^L \left [(\psth{\phi_k^i,\phi_m^i}-\psemp{\phi_k^i}{\phi_m^i})^2 + 2 (\psth{\phi_k^i,\phi_m^j}  - \psemp{\phi_k^i}{\phi_m^j})^2 \right.\\
%   &&\left.+(\psth{\phi_k^j,\phi_m^j}-\psemp{\phi_k^j}{\phi_m^j}) ^2 \right]
%   \end{array}
%  \eeqe
%     
%Pour chacune des différences, on peut appliquer une inégalité de type Bernstein. Néanmoins, il faut réintégrer $A^{-1}$. De la même manière, la quantité $\hat D_n-D$ est aussi facilement contrôlable, mais quid de $A^{-1}(\hat D_n-D)$ ?     
%     

A uniform bound for $\normmat{A^{-1}}$ (over all couples $(i,j)$) can be easily obtain since $A$ (and obviously $A^{-1}$) is Hermitian.
% and has a block-diagonal structure of the $p(p-1)$ submatrices $A^{ij}$. From this block-diagonal structure, we deduce that
$$
\normmat{A^{-1}} \leq \max_{(i',j') \in [1:p_n]^2} \rho\left(\left(A^{i'j'}\right)^{-1}\right)
$$
Simple algebra then yields
$$
\rho\left(\left(A^{i'j'}\right)^{-1}\right) \leq Tr \left(\left(A^{i'j'}\right)^{-1}\right) = \frac{Tr \left( Com(A^{i'j'})^t\right)}{det(A^{i'j'})}   = \frac{1}{det(A^{i'j'})} \sum_{k=1:2L} Com(A^{i'j'})_{k,k}
$$
where $Com(A^{ij})$ is the cofactor matrix associated to $A^{ij}$.
Now, recall the classical inequality (that can be found in \citet*{Bullen}): for any symetric definite positive matrix squared $S$ of size $Q \times Q$
$$
\det(S) \leq \prod_{\ell=1}^Q |S_{\ell\ell}|.
$$
This last inequality applied to the determinant involved in $Com(A^{i'j'})_{k,k}$
associated with $\mathbf{(H_{b}^1)}$  implies
$$
\forall k \in [1:2L] \qquad \left| Com(A^{i'j'})_{k,k}\right| \leq \{M^2\}^{2L-1}.
$$
We then deduce from $\mathbf{(H_{b}^{3,\vartheta})}$  that there exists a constant 
$C>0$ such that:

\beq \label{eq:invA}
\barr{lll}
\normmat{A^{-1}}
	&\leq &  \max_{(i,j) \in [1:p_n]^2} \frac{2L M^{4L-2}}{det(A^{i'j'})} \\
	&\leq & 2 C^{-1} L M^{4L-2} n^{\vartheta}.
\earr
\eeq

Similarly, if we denote $\Delta_{n_1} = A-\hat{A}_{n_1}$, we have
\beqe
\barr{lll}
\normmat{\left( \id -A^{-1}(A-\hat{A}_{n_1}) \right)^{-1}} &=  & \rho \left( \left( I-A^{-1}\Delta_{n_1}  \right)^{-1} \right)\\
	&= &\displaystyle\max_{\alpha \in Sp(A^{-1}\Delta_{n_1} )} \frac{1}{\left|1-\alpha\right|},
\earr
\eeqe
using the fact that $A-\hat{A}_{n_1}$ is self-adjoint.
We have seen that $\rho(A^{-1}) \leq 2  C^{-1} L M^{4L-2} n^{\vartheta}$ and Lemma \ref{lem:A} yields $\rho \left( \Delta_{n_1}\right) =\gopxi$. As a consequence, we have
$$
\displaystyle\max_{\alpha \in Sp(A^{-1}\Delta_{n_1} )} |\alpha| \leq \rho(A^{-1}) \rho\left( \Delta_{n_1}\right) = \gopxiv.
$$
At last, remark that
$$
\displaystyle\max_{\alpha \in Sp(A^{-1}\Delta_{n_1} )} \frac{1}{\left|1-\alpha\right|} - 1 = \displaystyle\max_{\alpha \in Sp(A^{-1}\Delta_{n_1} )} \frac{1-|1-\alpha|}{|1-\alpha|}
$$
We know that for $n$ large enough, each absolute value of $\alpha \in Sp(A^{-1}\Delta_{n_1} )$ becomes smaller than $1/2$ with a probability tending to one. Hence, we have with probability tending to one
$$
\displaystyle\max_{\alpha \in Sp(A^{-1}\Delta_{n_1} )}\left| \frac{1-|1-\alpha|}{|1-\alpha|} \right|\leq \displaystyle\max_{\alpha \in Sp(A^{-1}\Delta_{n_1} )}\frac{|\alpha|}{1-\alpha} \leq 2 \rho(A^{-1}\Delta_{n_1} ).
$$
Since $\rho(A^{-1}\Delta_{n_1} ) = \gopxiv$, we deduce
\beq \label{eq:inv}
\underset{i,j, \lij}{\sup} \normmat{\left( \id -A^{-1}(A-\hat{A}_{n_1}) \right)^{-1}} \leq 1+2LM^{4L-2} C^{-1} \gopxiv.
\eeq

To conclude the proof, we can now apply the same argument as the one  used in  Lemmas \ref{lem:A} and \ref{lem:D} with Bernstein's Inequality, using Equations (\ref{eq:invA}) and (\ref{eq:inv}).
\end{proof}

The last lemma finally compares the constant $\hat C^{n_1}$ with $C$.
\begin{lem}\label{lem:C}
Under Assumptions $\mathbf{(H_{b})}$, we have:
$$\underset{i,j, \lij }{\sup} \abs{ \hat C^{n_1}-C } = \gopxi.$$
\end{lem}

\begin{proof}
For any couple $(i,j)$, remark that constants $\hat C^{n_1}$ and $C$ satisfy:
$$C = -\psth{\pfi\times \pfj}{ 1} \ \ \mbox{and} \ \ \hat C^{n_1} = -\psempn{\pfi\times \pfj}{1}.$$

If we denote
$$
\Delta_{i,j, \lij } := \frac{1}{n_1}\sum_{r=1}^{n_1} \pfi(\xixi^r) \pfj(\xjxj^r)- \E(\pfi(X_i) \pfj(X_j)),
$$
we can apply again Bernstein's Inequality on  $(\pfi(\xixi^r) \pfj(\xjxj^r))_{r=1, \cdots,n_1}$. From $\mathbf{(H_{b}^1)}$, these 
 independent  random variables are bounded by $M^2$ and
\begin{eqnarray*}
P\left( \underset{i,j, \lij }{\sup} \abs{\Delta_{i,j,\lij} } \geq 
\gamma n_1^{-\xi/2} \right) 
& \leq&
\sum_{i,j,\bolds{l_{ij}}}  P\left( \abs{\Delta_{i,j, \lij }} \geq 
\gamma n_1^{-\xi/2} \right)  \\
& \leq& \sum_{i,j,\bolds{l_{ij}}} 2 \exp  \left( - \frac 1 2 \frac{\gamma^2 n_1^{1-\xi}}{M^4 +  M^2 \gamma/3 n_1^{-\xi/2}}\right)\\
&  \leq & 2 L^2 p_n^2  \exp  \left( - \frac 1 2 \frac{\gamma^2 n_1^{1-\xi}}{M^4 +  M^2 \gamma/3 n_1^{-\xi/2}}\right).
 \end{eqnarray*}

Under Assumption $\mathbf{(H_{b}^2)}$, the right-hand side of this inequality can be arbitrarly small for $n$ large enough, which ends the proof.
\end{proof}

To finish the proof of Theorem 1, remark that:
 
 \beqe
 \begin{array}{lll}
 \norm{\vfij-\pfij} &=& \norm{\sum_{k=1}^L ( \hat\lambda_{k,n_1}^i - \lambda_{k}^i) \phi_k^i + \sum_{k=1}^L ( \hat\lambda_{k,n_1}^j - \lambda_{k}^j) \phi_k^j + (\hat C^{n_1}-C)}\\
 &\leq & \underbrace{\norm{\sum_{k=1}^L ( \hat\lambda_{k,n_1}^i - \lambda_{k}^i) \phi_k^i + \sum_{k=1}^L ( \hat\lambda_{k,n_1}^j - \lambda_{k}^j) \phi_k^j }}_{I}+\abs{ \hat C^{n_1}-C}.
 \end{array}
 \eeqe
Moreover,  
 
 \beqe
 \begin{array}{lll}
 I^2 &=&  \int \left( \sum_{k=1}^L ( \hat\lambda_{k,n_1}^i - \lambda_{k}^i) \phi_k^i + \sum_{k=1}^L ( \hat\lambda_{k,n_1}^j - \lambda_{k}^j) \phi_k^j \right)^2 p_{X_i,X_j}(\xixi,\xjxj)dx_idx_j\\
 &=& \underbrace{\int \left(\sum_{k=1}^L ( \hat\lambda_{k,n_1}^i - \lambda_{k}^i) \phi_k^i \right)^2 p_{X_i}(\xixi)d x_i}_{I_1}+\underbrace{\int \left(\sum_{k=1}^L ( \hat\lambda_{k,n_1}^j - \lambda_{k}^j) \phi_k^j \right)^2 p_{X_j}(\xjxj)d x_j}_{I_2}\\
% &&  +\underbrace{\int \left(\sum_{k=1}^L ( \hat\lambda_{k,n}^j - \lambda_{k}^j) \phi_k^j \right)^2 p_{X_j}(\xjxj)d x_j}_{I_2}\\
 &&+ \underbrace{2   \int \left(\sum_{k=1}^L ( \hat\lambda_{k,n_1}^i - \lambda_{k}^i) \phi_k^i \right) \left(\sum_{k=1}^L ( \hat\lambda_{k,n_1}^i - \lambda_{k}^i) \phi_k^i \right) p_{X_i,X_j}(\xixi,\xjxj)d \xixi d \xjxj}_{I_3}.
 \end{array}
 \eeqe

 Using the inequality $2ab\leq a^2+b^2$, we thus deduce that $I_3 \leq I_1+ I_2$, and
 
 \beqe
 \begin{array}{lll}
 I_1 &=& \int \sum_{k=1}^L \sum_{m=1}^L ( \hat\lambda_{k,n_1}^i - \lambda_{k}^i) ( \hat\lambda_{m,n_1}^i - \lambda_{m}^i) \phi_k^i(\xixi) \phi_m^i(\xixi) p_{X_i}(\xixi)d x_i\\
 &=&  \sum_{k=1}^L  ( \hat\lambda_{k,n_1}^i - \lambda_{k}^i)^2 \quad\textrm{by orthonormality}.
 \end{array}
 \eeqe
 
And the same equality is satisfied for $I_2$: 
$
 I_2=\sum_{k=1}^L  ( \hat\lambda_{k,n_1}^j - \lambda_{k}^j)^2. 
$

Consequently, we obtain
  
  \beq\label{ineg}
  \begin{array}{lll}
  \norm{\vfij-\pfij} &\leq & \sqrt{ 2 \left[ \sum_{k=1}^L  ( \hat\lambda_{k,n_1}^i - \lambda_{k}^i)^2 +\sum_{k=1}^L  ( \hat\lambda_{k,n_1}^j - \lambda_{k}^j)^2  \right]} + \abs{ \hat C^{n_1}-C }\\
  &= & \sqrt{2} \eucl{\hat{\bolds\lambda}_{n_1} - \bolds\lambda}+ \abs{ \hat C^{n_1}-C }.
  \end{array}
  \eeq

The end of the proof follows with Lemmas \ref{lem:lambda} and \ref{lem:C}.   

\hfill $\square$

\noindent {\bf 6.4 Proof of Theorem 2\label{sec:theo2}}
%\section{Proof of Theorem 2}\label{sec:theo2}
%\setcounter{equation}{0}

We recall first that $\psth{}{}$ denotes the theoretical inner product based on the law $\PX$ (and $\norm{}$ is the derived Hilbertian norm). A careful inspection of the Gram-Schmidt procedure used to build the HOFD shows that
$$
M^* := 
\underset{u,\bolds{l_u}}{\sup} \norm{\pfu(\XuXu)}_{\infty} < \infty,
$$
provided that $\mathbf{(H_{b}^1)}$ holds.

 Now, remark that the EHOFD is obtained through the first sample ${\cal O}_1$ which determines the first empirical inner product $\psempn{}{}$ although the $\mathbb{L}^2$-boosting depends on the second sample ${\cal O}_2$.  Indeed, ${\cal O}_2$  determines the second empirical inner product $\psempnn{}{}$. Hence, $\psempnn{}{}$ uses observations which are \textit{independent} to the ones used to build the HOFD.

We begin this section with a lemma which establishes that the estimated functions $\vfu$ (which result in the EHOFD) are bounded.
\begin{lem}\label{lem:max}
 Under Assumption $\mathbf{(H_{b})}$, define
$$N_{n_1}:=\underset{u,\bolds{\bolds{l_u}}}{\sup} \norm{\vfu(\XuXu)}_{\infty}.$$
Then, we have:
$$
N_{n_1} - M^* = \gopxiv.
$$
\end{lem}

\begin{proof}
Using the decomposition of $\vfu$ on the dictionary, Assumption $\mathbf{(H_{b}^2)}$ and Cauchy-Schwarz Inequality, there exists a fixed constant $C>0$ such that for all $u \in S$, $\bolds{l_{u}}$:
$$
\forall x \in \mathbb{R}^p \qquad 
|\vfu(x) - \pfu(x)| \leq C M \sqrt{L} \sqrt{\eucl{\incchap-\inc}}+\norm{ \hat C_{\bolds{l_{u}}}^{n_1}-C_{\bolds{l_{u}}}}.
$$
The conclusion then follows using Lemmas \ref{lem:lambda} and \ref{lem:C}.
\end{proof}

We now present a key lemma which compares the elements $(\pfu)_{\bolds{\bolds{l_u}},u}$ with its estimated version $(\vfu)_{\bolds{\bolds{l_u}},u}$. 
\begin{lem}\label{lembernstein}
Assume that $\mathbf{(H_{b})}$ holds with $\xi \in (0,1)$, that the noise $\varepsilon$ satisfies $\mathbf{(H_{\varepsilon,q})}$ with $q>4/\xi$ and that $\mathbf{(H_{s})}$ is fullfilled. Then, the following equalities hold,

\begin{enumerate}[(i)]
\item \label{itemi} 
$$
\sup_{u,v,\bolds{\bolds{l_u}},\bolds{l_v}}|\psth{\vfu}{\vfv}-\psth{\pfu}{\pfv}|= \zeta_{n,1}=\gopxiv
$$
\item \label{itemii} 
$$\sup_{u,v,\bolds{\bolds{l_u}},\bolds{l_v}}|\psempnn{\vfu}{\vfv}-\psth{\pfu}{\pfv}|= \zeta_{n,2}=\gopxiv$$
\item \label{itemiii}
$$\sup_{u,v,\bolds{\bolds{l_u}},\bolds{l_v}}|\psempnn{\varepsilon}{\vfu}|=\zeta_{n,3}=\gopxi$$
\item \label{itemiv}
$$\underset{u,\bolds{\bolds{l_u}}}{\sup} \abs{\psempnn{\tilde{f}}{\vfu} - \psth{\tilde{f}}{\vfu}} = \zeta_{n,4} =\gopxi$$
\end{enumerate}

%\begin{itemize}
%\item[$(i)$]%\label{itemi}
%$$\sup_{u,v,\bolds{\bolds{l_u}},\bolds{l_v}}|\psth{\vfu}{\vfv}-\psth{\pfu}{\pfv}|= \zeta_{n,1}=\gopxiv$$
%\item[$(ii)$]
%$$\sup_{u,v,\bolds{\bolds{l_u}},\bolds{l_v}}|\psempnn{\vfu}{\vfv}-\psth{\pfu}{\pfv}|= \zeta_{n,2}=\gopxiv$$
%\item[$(iii)$]
%$$\sup_{u,v,\bolds{\bolds{l_u}},\bolds{l_v}}|\psempnn{\varepsilon}{\vfu}|=\zeta_{n,3}=\gopxi$$
%\item[$(iv)$]
%$$\underset{u,\bolds{\bolds{l_u}}}{\sup} \abs{\psempnn{\textcolor{red}{\tilde{f}}}{\vfu} - \psth{\textcolor{red}{\tilde{f}}}{\vfu}} = \zeta_{n,4} =\gopxi$$
%\end{itemize}

In the sequel, we will denote  $\zeta_n:=\max_{i\in \disc 0 4}\{ \zeta_{n,i} \}$.
\end{lem}

\begin{proof}

\textbf{Assertion $(\ref{itemi})$} 
Let $u,v \in S$, $\bolds{\bolds{l_u}}\in\disc 1 L^{|u|}$ and $\bolds{l_v} \in \disc 1 L^{|v|}$. Then, we have 
 \beqe
 \barr{lll}
 \abs{\psth{\vfu}{\vfv}-\psth{\pfu}{\pfv}} 
&\leq& \abs{ \psth{\vfu - \pfu}{\vfv}-\psth{\pfu}{\pfv- \vfv}}\\[0.2cm]
 &\leq& \norm{\vfu-\pfu} \norm{ \vfv}+\norm{\pfu} \norm{ \vfv-\pfv}\\[0.2cm]
 &\leq &  \norm{\vfu-\pfu}  \left(\norm{\vfv -\pfv}+ 1\right )+ \norm{ \vfv-\pfv},
 \earr
 \eeqe
 and the conclusion holds applying Theorem 1.
  
\textbf{Assertion $(\ref{itemii})$} We breakdown it in two parts: 
  \beqe
 \barr{lll}
 \abs{\psempnn{\vfu}{\vfv}-\psth{\pfu}{\pfv}} &\leq & \underbrace{\abs{\psempnn{\vfu}{\vfv}-\psth{\vfu}{\vfv}} }_{I}\\[0.2cm] & & +\underbrace{ \abs{\psth{\vfu}{\vfv}-\psth{\pfu}{\pfv}}}_{II}.
  \earr
 \eeqe
 
Assertion $(\ref{itemi})$ implies that,
$$\underset{u,v,\bolds{\bolds{l_u}},\bolds{l_v}}{\sup} |II| =\gopxiv.$$
To control $\underset{u,v,\bolds{\bolds{l_u}},\bolds{l_v}}{\sup} |I|$, we use Bernstein's inequality to the family of independent random variables $\left( \vfu(\xuxu^s) \vfv(\xvxv^s)  \right)_{s=1...n_2}$ and we denote
$$
\Delta_{u,v,\lu,\bolds{l_v}}= \abs{\frac{1}{n_2} \sum_{s=1}^{n_2} \vfu(\xuxu^s) \vfv(\xvxv^s) - \mathbb{E}( \vfu(\XuXu) \vfv(\XvXv) )}
$$
Then, Bernstein's inequality implies that
\begin{eqnarray*}
 P\left( \underset{u,v,\bolds{\bolds{l_u}},\bolds{l_v}}{\sup} \Delta_{u,v,\bolds{\bolds{l_u}},\bolds{l_v}}\geq \gamma n_2^{-\xi/2} \right)  & \leq & 
 P\left( \underset{u,v,\bolds{\bolds{l_u}},\bolds{l_v}}{\sup} \Delta_{u,v,\bolds{\bolds{l_u}},\bolds{l_v}}\geq \gamma n_2^{-\xi/2} \& N_{n_1} < M^*+1\right)  \\
 & &  +  P\left( \underset{u,v,\bolds{\bolds{l_u}},\bolds{l_v}}{\sup} \Delta_{u,v,\bolds{\bolds{l_u}},\bolds{l_v}}\geq \gamma n_2^{-\xi/2} \& N_{n_1} > M^*+1\right)  \\
 & \leq &   64 L^4 p_n^4 \exp \left( - \frac 1 2 \frac{\gamma^2 {n_2}^{1-\xi}}{(M^*+1)^4 + (M^*+1)^2 \gamma/3 {n_2}^{-\xi/2})} \right) \\ & &+ P\left( N_{n_1} > M^*+1\right) 
\end{eqnarray*}

%\red{pourquoi 64 ? C'est pas 4 ?}
Lemma \ref{lem:max} and Assumption $\mathbf{(H_{b}^2)}$ allows for deducing  $(\ref{itemii})$.

\textbf{Assertion $(\ref{itemiii})$}  The proof follows the roadmap of $(\ref{itemii})$ of Lemma 1 of ~\citet*{buhlmann06}. We thus define the truncated variable $\varepsilon_{t}$ for all $s\in \disc 1 {n_2}$,
\beqe
\varepsilon^s_{t}=\left\{
\begin{array}{ll}
\varepsilon^s & \text{if }~|\varepsilon^s | \leq K_n\\
sg(\varepsilon^s) K_n &  \text{if }~|\varepsilon^s| > K_n
\end{array}
\right.
\eeqe

where $sg(\varepsilon)$ denotes the sign of $\varepsilon$. Then, for $\gamma>0$, we have:

\beqe
\begin{array}{lll}
P\left(n_2^{\xi/2} \underset{u,\bolds{\bolds{l_u}}}{\sup}\abs{\psempnn{\vfu}{\varepsilon}} > \gamma \right) & \leq & P\left(n_2^{\xi/2}  \underset{u,\bolds{\bolds{l_u}}}{\sup} \abs{\psempnn{\vfu}{\varepsilon_t} - \psth{\vfu}{\varepsilon_t}}>\gamma/3\right) \\
&&+ P\left(n_2^{\xi/2}  \underset{u,\bolds{\bolds{l_u}}}{\sup}\abs{\psempnn{\vfu}{\varepsilon-\varepsilon_t}} >\gamma/3\right) \\
&& + P\left(n_2^{\xi/2}  \underset{u,\bolds{\bolds{l_u}}}{\sup}\abs{\psth{\vfu}{\varepsilon_t}}>\gamma/3\right) \\
&=& I+II+III
\end{array}
\eeqe
\underline{Term $II$:}
We can bound $II$ using the following simple inclusion:
\begin{eqnarray*}
\left\{n_2^{\xi/2}  \underset{u,\bolds{\bolds{l_u}}}{\sup} \abs{\psempnn{\vfu}{\varepsilon_t} - \psth{\vfu}{\varepsilon_t}}>\gamma/3\right\}
&\subset &\left\{\text{there exists s such that} \,  \varepsilon^s - \varepsilon_{t}^s \neq 0  \right\}\\
& = & \left\{\text{there exists s such that} \, | \varepsilon^s| > K_n  \right\}\\
\end{eqnarray*}

Hence, 
\begin{eqnarray*}
II&\leq& P(\text{some}~\abs{\varepsilon^s}> K_n )\\
& \leq&  n_2 P(\abs \varepsilon > K_n)\leq n_2 K_n^{-q} \E(\abs \varepsilon ^q) =  \underset{n\rightarrow +\infty}{\mathcal{O}}(n^{1-q\xi/4}), \\
\end{eqnarray*}
where $n_2=n/2$ and we have chosen $K_n := n^{\xi/4}$ since $q>4/ \xi$ by  Assumption of the Lemma. Hence, $II$ can become arbitrarily small.\\

\underline{Term $I$:}
Using again Bernstein's Inequality to the family of independent random variables $(\vfu(\xuxu^s) \varepsilon_t^s )_{s=1, \cdots, n_2}$ and considering the two  events  
$\{N_{n_1}>M^*+1\}$ and $\{N_{n_1}<M^*+1\}$, we can also show that:
\beqe
I \leq 2 L p_n \exp\left( - \frac 1 2 \frac{(\gamma^2/9) {n_2}^{1-\xi}}{(M^*+1)^4 \sigma^2 + (M^*+1)K_n \gamma/9 {n_2}^{-\xi/2}} \right)+ P(N_{n_1}>M^*+1) ,
\eeqe
where $\sigma^2 := \E(\abs{\varepsilon}^2).$ 
We can then make the right-hand side of the previous inequality arbitrarily small owing to $\mathbf{(H_{b}^2)}$  with $K_n=n^{\xi/2}$.

 \underline{Term $III$:} by assumption, $\E(\pfu(\XuXu)\varepsilon)=0$. We then have:
 
\beqe
\begin{array}{lll}
III &\leq&  P\left(n_2^{\xi/2} \underset{u,\bolds{\bolds{l_u}}}{\sup}\abs{\E[ (\vfu-\pfu)(\XuXu) \varepsilon_t] }>\gamma/6\right)+ P\left(n_2^{\xi/2}\underset{u,\bolds{\bolds{l_u}}}{\sup} \abs{\E[ \pfu(\XuXu) (\varepsilon-\varepsilon_t)] }>\gamma/6\right)\\
&=& III_1 + III_2,
\end{array}
\eeqe
with,
\beqe
\begin{array}{lll}
III_1 &=&  P\left( n_2^{\xi/2} \underset{u,\bolds{\bolds{l_u}}}{\sup}  \abs{\E[ (\vfu-\pfu)(\XuXu)]} \abs{\E(\varepsilon_t) }>\gamma/6\right)\\
&\leq & P\left(n_2^{\xi/2}\underset{u,\bolds{\bolds{l_u}}}{\sup}  \abs{\E[ (\vfu-\pfu)(\XuXu)]} \abs{\E(\varepsilon_t) }>\gamma/6\right)\\
&\leq &  \mathds 1_{\{n_2^{\xi/2} \underset{u,\bolds{\bolds{l_u}}}{\sup}  \abs{\E[ (\vfu-\pfu)(\XuXu)]} \abs{\E(\varepsilon_t)} >\gamma/6\}}\\
\end{array}
\eeqe
Moreover, one has
\beq\label{bounderror}
\begin{array}{lll}
\abs{\E(\varepsilon_t)} &=& \abs{\int_{\abs x \leq K_n} x \dd \Pep(x)+ \int_{\abs x > K_n} sg(x) K_n \dd \Pep(x)}
=\abs{  \int_{\abs x > K_n} (sg(x) K_n -x) \dd \Pep(x)} \\[0.2cm]
&\leq & \int \mathds 1_{\abs x >K_n} (K_n+ \abs x) \dd \Pep(x)\\[0.2cm]
&\leq & K_n \Pep(\abs \varepsilon > K_n) + \int \abs x \mathds 1_{\abs x >K_n} \dd \Pep(x)\\[0.2cm]
&\leq & K_n^{1-t} \E(\abs \varepsilon ^t)+ \E(\varepsilon^2)^{1/2} K_n^{-t/2} \E(\abs \varepsilon ^t)^{1/2} \quad \text{by  the Tchebychev Inequality }\\[0.2cm]
&\leq &  O(K_n^{1-t})+  O(K_n^{-t/2})=o(K_n^{-2})
\end{array}
\eeq

since $0<\xi< 1$ and $t >4/\xi>4$.  Then, set $K_n=n^{\xi/4}$, we obtain:
\beqe
n_2^{\xi/2} \norm{ \vfu-\pfu}  \abs{\E(\varepsilon_t)}  \leq   n_2^{\xi/2} o(1) o(n^{-\xi/2}) = o(1),
\eeqe
when $o$ is the usual Landau notation of relative insignificance.

Hence, $III_1=0$ for $n$ large enough. For $III_2$, one has

\beqe
III_2 \leq \mathds 1_{\{ n_2^{\xi/2}\underset{u,\bolds{\bolds{l_u}}}{\sup} \abs{\E[ \pfu(\XuXu) (\varepsilon-\varepsilon_t)] }>\gamma/6  \}},
\eeqe
and, by independance, 

\beqe
\abs{\E[ \pfu(\XuXu) (\varepsilon-\varepsilon_t)] }=\abs{\E[ \pfu(\XuXu)]}\abs{ \E(\varepsilon-\varepsilon_t )} \leq M^* \abs{ \E(\varepsilon-\varepsilon_t )}. 
\eeqe

Equation (\ref{bounderror}) then implies,

\beqe
\abs{\E(\varepsilon-\varepsilon_t)} = \abs{  \int_{\abs x > K_n} (sg(x) K_n -x) \dd \Pep(x)} \leq o(K_n^{-2})= o(n^{-\xi/2})
\eeqe

Thus, $III$ is arbitrarily small for $n$ and $\gamma$ large enough and  $(\ref{itemiii})$ holds.

\textbf{Assertion $(\ref{itemiv})$}
Remark that,
\beqe
\underset{u,\bolds{\bolds{l_u}}}{\sup} \abs{\psempnn{\tilde{f}}{\vfu} - \psth{\tilde{f}}{\vfu}} \leq \betasumm \underset{u,\bolds{\bolds{l_u}}}{\sup}  \abs{\psempnn{\pfv}{\vfu}-\psth{\pfv}{\vfu}}.
\eeqe
Now, $\mathbf{(H_{s})}$ and  Bernstein's Inequality implies
\begin{eqnarray*}
P\left( \underset{u,\bolds{\bolds{l_u}}}{\sup}  \abs{\psempnn{\pfv}{\vfu}-\psth{\pfv}{\vfu}} \geq \gamma n_2^{-\xi/2} \right) &\leq & P(N_{n_1}>M^*+1)  \\ 
+2 L p_n \exp \left( - \frac 1 2 \frac{\gamma^2 {n_2}^{1-\xi}}{(M^*+1)^4 +  (M^*+1)^2  \gamma/3 {n_2}^{-\xi/2}} \right),&&\\
\end{eqnarray*}
which implies with Assumption $\mathbf{(H_{b}^2)}$ that:
$$\underset{u,\bolds{\bolds{l_u}}}{\sup}  \abs{\psempnn{\pfv}{\vfu}-\psth{\pfv}{\vfu}} = \gopxi.$$

 \end{proof}
 
 The following lemma, similar to Lemma 2 from~\citet*{buhlmann06}, then holds:
  \begin{lem} \label{lemme:lemmeBuhl}
Under Assumptions $\mathbf{(H_{b})}$, $\mathbf{(H_{\varepsilon,q})}$ with $q > 4/\xi$  and $\mathbf{(H_{s})}$, there exists a constant $C>0$ such that, on the set $\Omega_n=\{\omega, |\zeta_n(\omega)|<1/2\} $:
$$ \underset{u,\bolds{\bolds{l_u}}}{\sup} |\psempnn { Y- G_{k}(\bar f)}{\vfu} - \psth{\tilde{R}_{k}(\bar{f})}{\pfu} | 
\leq C \left( \frac{5}{2} \right) ^{k} \zeta_n. $$
\end{lem}

\begin{proof}
Denote
$ A_n(k,u)=\psempnn { Y- G_{k}(\bar{f})}{\vfu} - \psth{\tilde{R}_{k}(\bar{f})}{\pfu}$. 
Assume first that $k=0$,

\begin{eqnarray*}
\underset{u,\bolds{\bolds{l_u}}}{\sup}|A_n(0,u)|&=&\underset{u}{\sup}|\psempnn { Y}{\vfu} - \psth{\bar{f}}{\pfu}|  \\
&\leq & \underset{u,\bolds{\bolds{l_u}}}{\sup} \left\{ \abs{ \psempnn{\tilde f}{\vfu} - \psth{\tilde f}{\vfu}} + \abs{ \psth{\tilde f-\bar f}{\vfu}  }  + \abs{ \psth{\bar f}{ \vfu-\pfu}  }\right\}\\
&& +\underset{u,\bolds{\bolds{l_u}}}{\sup} \abs{\psempnn{\varepsilon}{\vfu}}\\
&\leq & (3+  \norm{\bar{f}})\zeta_n \quad \textrm{by (\ref{itemiii})-(\ref{itemiv}) of Lemma \ref{lembernstein} and Theorem 1 }
%    &\leq& \underset{u,\bolds{\bolds{l_u}}}{\sup} |\psempnn { \bar{f}}{\vfu} - \psth{\bar{f}}{\vfu} | + | \psth{\bar{f}}{\vfu - \pfu}| + |\psempnn{\varepsilon}{\vfu}|\\
%    &\leq& \zeta_n + \norm{\bar{f}} \zeta_n + \zeta_n = (2+ \norm{\bar{f}} ) \zeta_n.
    \end{eqnarray*}
%The inequality follows directly from Theorem 1 and $(\ref{itemiii})-(\ref{itemiv})$ of Lemma \ref{lembernstein}.
From the main document, we remind that
\beq\label{hatg}
 G_k(\bar f)= G_{k-1}(\bar f)+\gamma \psempnn { Y- G_{k-1}(\bar f)}{\vfs} \cdot \vfs,
\eeq
\beq\label{theorres}
\barr{lll}
R_k(\bar f)&=&\bar f- G_k(\bar f)\\
&=& \bar f- G_{k-1}(\bar f)-\gamma \psempnn { Y- G_{k-1}(\bar f)}{\vfs} \cdot \vfs
%&=& R_{k-1}(\bar f)-\gamma \psempnn { \tilde f -\bar f+\bar f+\varepsilon- G_{k-1}(\bar f)}{\vfs} \cdot \vfs\\
%&=&  R_{k-1}(\bar f)- \gamma \psempnn{R_{k-1}(\bar f)}{\vfs} \vfs - \gamma \psempnn { \varepsilon}{\vfs} \cdot \vfs \\
%&&- \gamma \psempnn { \tilde f -\bar f}{\vfs} \cdot \vfs.
\earr
\eeq
and
\beq\label{phantom}
\left\{
\barr{l}
\tilde R_0(\bar f)=\bar f\\
\tilde R_k(\bar f)=\tilde R_{k-1}(\bar f)-\gamma \psth{\tilde R_{k-1}(\bar f)}{\vfs} \vfs.
\earr
\right.
\eeq
From the recursive relations (\ref{hatg}) and (\ref{phantom}), for any $ k \geq 0$, we obtain:
\begin{eqnarray*}
A_n(k,u) & = & \psemp { Y- G_{k-1}(\bar{f})-\gamma \psempnn { Y- G_{k-1}(\bar{f})}{\vfs} \cdot \vfs}{\vfu} \\
      & & - \psth{\tilde{R}_{k-1}(\bar{f}) - \gamma \psth{\tilde{R}_{k-1}(\bar{f})}{\vfs} \vfs}{\pfu} \\
%      &\leq& A_n(k-1,u) -\gamma \psempnn { Y-G_{k-1}(\bar{f})}{\vfs}\psempnn{\vfs}{\vfu} \\ 
%      &&+ \gamma \psth{\tilde{R}_{k-1}(\bar{f})}{\vfs} \psth{\vfs}{\pfu} \\    
%      &\leq& A_n(k-1,u)  -\gamma \psempnn { Y- G_{k-1}(\bar{f})}{\vfs}\psempnn{\vfs}{\vfu} \\
%      & & + \gamma \psth{\tilde{R}_{k-1}(\bar{f})}{ \pfs - (\pfs - \vfs)} \psth{\vfs}{\pfu} \\    
      &\leq &  A_n(k-1,u) \\ 
      & &  - \gamma \underbrace{\left(\psempnn { Y- G_{k-1}(\bar{f})}{\vfs}- \psth{\tilde{R}_{k-1}(\bar{f})}{\pfs} \right)\psempnn{\vfs}{\vfu}}_{I}\\
      & & + \gamma \underbrace{\psth{\tilde{R}_{k-1}(\bar{f})}{\pfs}\left( \psth{\vfs}{\pfu} - \psempnn{\vfs}{\vfu} \right)}_{II} \\
      & & + \gamma \underbrace{\psth{\tilde{R}_{k-1}(\bar{f})}{\vfs-\pfs} \psth{\vfs}{\pfu}}_{III}.
      \end{eqnarray*}
On the one hand, using assertion $(\ref{itemii})$ of Lemma \ref{lembernstein}, and the Cauchy-Schwarz inequality (with $\norm{\pfu}=1$), it comes
\begin{eqnarray}
\underset{u,\bolds{\bolds{l_u}}}{\sup} |I| & \leq &  \underset{u,\bolds{\bolds{l_u}}}{\sup} | \psempnn{\vfs}{\vfu}| \underset{u,\bolds{\bolds{l_u}}}{\sup} |A_n(k-1,u)| \nonumber\\
                                          & \leq & (\underset{u,\bolds{\bolds{l_u}}}{\sup} |\psth{\pfs}{\pfu}| + \zeta_n) \underset{u,\bolds{\bolds{l_u}}}{\sup} |A_n(k-1,u)|\nonumber\\
                                          & \leq & (1 + \zeta_n ) \underset{u,\bolds{\bolds{l_u}}}{\sup} |A_n(k-1,u)|. \nonumber
\end{eqnarray}
Consider now the phantom residual, from its recursive relation, we can show that $ \norm{ \tilde{R}_k(\bar{f})}^2 = 
\norm{\tilde{R}_{k-1}(\bar{f})}^2 - \gamma (2-\gamma) \psth{\tilde{R}_{k-1}(\bar{f})}{\vfs}^2 \leq \norm{\tilde{R}_{k-1}(\bar{f})}^2 $ and we deduce
\beq\label{rtilde}
\norm{ \tilde{R}_k(\bar{f})}^2 
\leq  \norm{\bar{f}}^2. 
\eeq
Then,
\begin{eqnarray*}
\underset{u,\bolds{\bolds{l_u}}}{\sup} |II| &\leq& \norm{\tilde{R}_{k-1}(\bar{f}) } \norm{\pfs} \underset{u,\bolds{\bolds{l_u}}}{\sup} | \psth{\vfs}{\pfu} - \psempnn{\vfs}{\vfu}|\\
  &\leq&   \norm{\bar{f}} \underset{u,\bolds{\bolds{l_u}}}{\sup} | \psth{ \vfs}{\pfu} - \psempnn{\vfs}{\vfu}|,
\end{eqnarray*}
with
\begin{eqnarray*}
 | \psth{ \vfs}{\pfu} - \psempnn{\vfs}{\vfu}| &\leq& | \psempnn{\vfs}{\vfu} - \psth{\pfs}{\pfu}|  \\
 & & + | \psth{\pfs-\vfs}{\pfu} |.
\end{eqnarray*}
Using again assertion $(\ref{itemii})$ from Lemma \ref{lembernstein} and Theorem 1, we obtain the following bound for II,
\begin{eqnarray*}
 \underset{u,\bolds{\bolds{l_u}}}{\sup} |II| &\leq&  \norm{\bar{f}} (\zeta_n +   \underset{u,\bolds{\bolds{l_u}}}{\sup}  \norm{\pfu-\vfu})\\
      &\leq& 2 \zeta_n  \norm{\bar{f}}.
\end{eqnarray*}
Finally,  Theorem 1 gives
\begin{eqnarray*}
\underset{u,\bolds{\bolds{l_u}}}{\sup} |III| &\leq& \underset{u,\bolds{\bolds{l_u}}}{\sup} \norm{\tilde{R}_{k-1}(\bar{f})} \norm{\vfs-\pfs} \norm{\vfs} \norm{ \pfu}\\
    &\leq &  \norm{\bar{f}} \zeta_n.
\end{eqnarray*}
Our bounds on $ I $, $ II $ and $ III $, 
and $ \gamma <1 $ yields on $ \Omega_n=\{ \zeta_n<1/2 \}$ that
\begin{eqnarray*}
\underset{u,\bolds{\bolds{l_u}}}{\sup} |A_n(k,u)| & \leq & \underset{u,\bolds{\bolds{l_u}}}{\sup} |A_n(k-1,u)| +(1+\zeta_n)\underset{u,\bolds{\bolds{l_u}}}{\sup} |A_n(k-1,u)|+3\zeta_n  \norm{\bar{f}}\\
	                                          & \leq & \frac{5}{2} \underset{u,\bolds{\bolds{l_u}}}{\sup} |A_n(k-1,u)|  + 3\zeta_n \norm{\bar{f}}.
\end{eqnarray*}
A simple induction yields:
\begin{eqnarray*}
\underset{u,\bolds{\bolds{l_u}}}{\sup} |A_n(k,u)| & \leq &  \left( \frac{5}{2} \right)^k \underbrace{\underset{u,\bolds{\bolds{l_u}}}{\sup} |A_n(0,u)|}_{\leq (3+ \norm{\bar{f}}) \zeta_n} + 3 \zeta_n  \norm{\bar{f}} \sum_{\ell=0}^{k-1} \left(\frac{5}{2}\right)^{\ell}\\
	      & \leq & \left(\frac{5}{2}\right)^k \zeta_n \left(3+ \betasumm%\sum_{u\in S}\sum_{\bolds{\bolds{l_u}}} |\beta_{\bolds{\bolds{l_u}}}^{u}| 
	      \left(1+ 3\sum_{\ell=1}^{\infty} \left(\frac{5}{2}\right)^{-\ell}\right) \right),
\end{eqnarray*}
which ends the proof with $ C=3+\betasumm \left(1+ 3\sum_{\ell=1}^{\infty} \left(\frac{5}{2}\right)^{-\ell}\right)$.

\end{proof}

 %%%%%%%%%%%%%%%%%%%% preuve du theoreme %%%%%%%%%%%%%%%%%%%%%%

We then aim at applying Theorem 2.1 from~\citet*{champion} to the phantom residuals $(\tilde R_k(\bar{f}))_k$. Using the notation of~\citet*{champion}, this will be possible if we can show that the phantom residuals follows a theoretical boosting with a shrinkage parameter $\nu \in [0,1]$. Thanks to Lemma \ref{lemme:lemmeBuhl} and by definiton
of $ \vfs $, one has
\begin{align}
|\psempnn {Y-G_{k-1}(\bar{f})}{\vfs} | & = 
\underset{u,\bolds{\bolds{l_u}}}{\sup} |\psempnn { Y- G_{k-1}(\bar{f})}{\vfu} | \nonumber \\
                                                                   & \geq 
\underset{u,\bolds{\bolds{l_u}}}{\sup} \left\{ |\psth{\tilde{R}_{k-1}(\bar{f})}{\pfu} |-C\left( \frac{5}{2} \right)^{k-1} \zeta_n \right\}. \label{k1}
\end{align}
%\begin{equation}
%|\langle Y - \hat{G}_{k-1}(f),\varphi_k\rangle_{(n)} |  = 
%\underset{1\leq j\leq p_n}{\sup} |\langle Y - \hat{G}_{k-1}(f),g_j\rangle_{(n)}|
%                                                                    = 
%\underset{1\leq j\leq p_n}{\sup} \left\{ |\langle\tilde{R}_{k-1}(f),g_j\rangle|-C\left( \frac{5}{2} \right)^{k-1} \zeta_n \right\}. \label{k1} 
%\end{equation}

Applying again Lemma \ref{lemme:lemmeBuhl} on the set $\Omega_n$, we obtain:
\begin{align}
|\psth{\tilde{R}_{k-1}(\bar{f})}{\pfs} | & \geq 
|\psempnn {Y-G_{k-1}(\bar{f})}{\vfs} | - 
C\left( \frac{5}{2} \right)^{k-1} \zeta_n \nonumber \\
                                                   & \geq 
\underset{u,\bolds{\bolds{l_u}}}{\sup} |\psth{\tilde{R}_{k-1}(\bar{f})}{\pfu} | - 
2C \left( \frac{5}{2} \right)^{k-1} \zeta_n. \label{l1}
\end{align}
Consider now the set $ \tilde{\Omega}_n=\left\{\omega, \quad \forall k \leq k_n, \quad  \underset{u,\bolds{\bolds{l_u}}}
{\sup} |\psth{\tilde{R}_{k-1}(\bar{f})}{\pfu} | > 4 C \left( \frac{5}{2} \right)^{k-1} \zeta_n\right\} $. 
We deduce from Equation (\ref{l1}) the following inequality on $\Omega_n \cap \tilde{\Omega}_n$:
\begin{equation} \label{pl}
|\psth{\tilde{R}_{k-1}(\bar{f})}{\pfs} | \geq 
\frac{1}{2}\underset{u,\bolds{\bolds{l_u}}}{\sup} |\psth{\tilde{R}_{k-1}(\bar{f})}{\pfu}|.
\end{equation}
Consequently, on $ \Omega_n\cap \tilde{\Omega}_n$, the family
$ (\tilde{R}_k(\bar f))_k $ satisfies a theoretical boosting, given by Algorithm 1 of~\citet*{champion}, with constant $\nu=1/2$ and we have: 
\begin{equation} \label{omega}
\norm{\tilde{R}_k(\bar{f})} \leq C' \left( 1+ \frac{1}{4} \gamma ( 2-\gamma)k\right)^{-\frac{2-\gamma}{2(6-\gamma)}}.
\end{equation}

Consider now the complementary set $$ \tilde{\Omega}_n^C=\left\{\omega, \quad  \exists\,  k \leq k_n  \quad \underset{ u,\bolds{\bolds{l_u}}}{\sup} |\psth{\tilde{R}_{k-1}(\bar{f})}{\pf{\bolds{l_u}} u}|\leq 4 C \left( \frac{5}{2} \right)^{k-1} \zeta_n\right\}.$$
Remark that
%\red{Ds $\norm{\tilde{R}_k(\bar{f})}^2 $, c'est $\vfu$ au lieu de $\pfu$}
\beqe
\barr{lll}
\norm{\tilde{R}_k(\bar{f})}^2 &=& \psth{\tilde{R}_k(\bar{f})}{\bar{f}-\gamma \sum_{j=0}^{k-1} \psth{\tilde{R}_j(\bar{f})}{\vfsbis} \vfsbis} \\
	&\leq &  \betasumm  \underset{ u,\bolds{\bolds{l_u}}}{\sup} \abs{ \psth{\tilde{R}_k(\bar{f})}{\vfu}} + \gamma \sum_{j=0}^{k-1} \abs{\psth{\tilde{R}_j(\bar{f})}{\vfsbis}} \underset{ u,\bolds{\bolds{l_u}}}{\sup} \abs{ \psth{\tilde{R}_k(\bar{f})}{\vfu}}.
\earr
\eeqe
Moreover, 
\beqe
\barr{lll}
\underset{ u,\bolds{\bolds{l_u}}}{\sup} \abs{ \psth{\tilde{R}_k(\bar{f})}{\vfu}} &\leq &\underset{ u,\bolds{\bolds{l_u}}}{\sup} \abs{ \psth{\tilde{R}_k(\bar{f})}{\pfu}} +  \underset{ u,\bolds{\bolds{l_u}}}{\sup} \abs{ \psth{\tilde{R}_k(\bar{f})}{\vfu - \pfu}}\\
&\leq & \underset{ u,\bolds{\bolds{l_u}}}{\sup} \abs{ \psth{\tilde{R}_k(\bar{f})}{\pfu}} +  \norm{\bar{f}} \zeta_n \quad \textrm{by Theorem 1 and (\ref{rtilde})}
\earr
\eeqe
%with $\underset{ u,\bolds{\bolds{l_u}}}{\sup} \abs{ \psth{\tilde{R}_k(\bar{f})}{\vfu - \pfu}} \leq \norm{\bar{f}} \zeta_n.$

We hence have
\beq\label{eq:eq1}
\barr{lll}
\norm{\tilde{R}_k(\bar{f})}^2 &\leq & \left( \betasumm + \gamma \sum_{j=0}^{k-1} \abs{\psth{\tilde{R}_j(\bar{f})}{\vfsbis}} \right) \left( \underset{ u,\bolds{\bolds{l_u}}}{\sup} \abs{ \psth{\tilde{R}_k(\bar{f})}{\pfu}} +\norm{\bar{f}} \zeta_n  \right)\\ 
	&\leq & \left( \betasumm + \gamma k\norm{\bar{f}} \right) \left( \underset{ u,\bolds{\bolds{l_u}}}{\sup} \abs{ \psth{\tilde{R}_k(\bar{f})}{\pfu}} +\norm{\bar{f}} \zeta_n \right)\\
	&\leq &  \left( \betasumm + \gamma k\norm{\bar{f}} \right) \left(4 C \left( \frac{5}{2} \right)^{k} \zeta_n +\norm{\bar{f}} \zeta_n \right) \quad \textrm{on }~\tilde{\Omega}_n^C
\earr
\eeq

%We thus deduce the following inequality on the set $\tilde{\Omega}_n^C$:
%\beq\label{eq:eq1}
% \norm{ \tilde{R}_k(\bar{f})}^2 \leq \left( \betasumm + \gamma k\norm{\bar{f}} \right)4 C \left( \frac{5}{2} \right)^{k} \zeta_n+ \gamma k \norm{\bar{f}}^2 \zeta_n.  
%\eeq

Finally, on the set $ (\Omega_n \cap \tilde{\Omega}_n) \cup \tilde{\Omega}_n^C$, by Equations (\ref{omega}) and (\ref{eq:eq1}),
\beq\label{eer1}
\norm{ \tilde{R}_k(\bar{f})}^2  \leq  
C'^2 \left( 1+\frac{1}{4} \gamma (2-\gamma)k \right) ^{-\frac{2-\gamma}{6-\gamma}} 
+ \left( \betasumm + \gamma k\norm{\bar{f}} \right) (4 C \left( \frac{5}{2} \right)^{k} \zeta_n +\norm{\bar{f}} \zeta_n) %+ \gamma k \norm{\bar{f}}^2 \zeta_n.  
\eeq

To conclude the first part of the proof, remark that $$ P\left((\Omega_n \cap \tilde{\Omega}_n) \cup \tilde{\Omega}_n^C \right) \geq P(\Omega_n) \underset{n \rightarrow +\infty}{\longrightarrow} 1 .$$ 
Now, by Assumption $(\mathbf{H_s})$ and by Lemma \ref{lem:max}, we have,
\beqe
\norm{\bar{f}} \zeta_n \leq \betasumm N_{n_1} \zeta_n \leq \betasumm (M^*+\gopxiv) \zeta_n \rightarrow 0.
\eeqe

Thus, Inequality (\ref{eer1}) holds almost surely, and for $ k_n< (\xi/2-\vartheta)/2 \log(3) \log(n) $, which grows sufficiently slowly, we get
\begin{equation}\label{eq:final}
\norm{\tilde{R}_{k_{n}}(\bar{f})} \cvproinf 0.
\end{equation}

Consider now $ A_k:=\norm{R_k(\bar f)  - \tilde{R}_k(\bar{f}) } $ for $k\geq 1$.
By definitions reminded in (\ref{theorres})-(\ref{phantom}), we have:

%\begin{eqnarray*}
%R_k(\bar f) - \tilde{R}_k(\bar{f}) &= &\bar{f} - G_{k-1}(\bar{f})-\gamma \psempnn{Y-G_{k-1}(\bar{f})}{\vfs} \vfs\\
%& &  - \left( \tilde{R}_{k-1}(\bar{f})
%- \gamma \psth{\tilde{R}_{k-1}(\bar{f})}{\vfs}\vfs \right)
%\end{eqnarray*}
%Thus, 
\begin{eqnarray}
A_k  	  &\leq& A_{k-1} + \gamma |\psempnn{Y-G_{k-1}(\bar{f})}{\vfs}-\psth{\tilde{R}_{k-1}(\bar{f})}{\vfs}|\nonumber \\
 	  &\leq&  A_{k-1} + \gamma |\psempnn{Y-G_{k-1}(\bar{f})}{\vfs}-\psth{\tilde{R}_{k-1}(\bar{f})}{\pfs}| \label{ineq1} \\
 	  & &+ \gamma |\psth{\tilde{R}_{k-1}(\bar{f})}{\vfs -\pfs}| \nonumber.
\end{eqnarray}
By Lemma \ref{lemme:lemmeBuhl}, we then deduce the following inequality on $ \Omega_n $:
\beq \label{aa1}
A_k \leq A_{k-1} + \gamma\left( C \left( \frac{5}{2} \right) ^{k-1} +1 \right) \zeta_n + \gamma \norm{\bar{f}} \zeta_n. 
\eeq

Since $A_0=0$, we deduce recursively from Equation (\ref{aa1}) that, on $\Omega_n$, 
%$$ A_k \leq A_{0}+ \gamma \left( C \left( \frac{5}{2} \right) ^{k-1} +1 \right) k \zeta_n, $$
%with $ A_0=0 $.
\beqe
 A_{k_{n}} \cvproinf 0.
\eeqe

Finally, as
\beqe\label{decomporesbis}
\norm{\hat f-\tilde f}=\norm{G_{k_{n}}(\bar f)-\tilde f} \leq \norm{\bar f-\tilde f} + \norm{R_{k_{n}}(\bar f)-\tilde  R_{k_{n}}(\bar f)}+ \norm{\tilde  R_{k_{n}}(\bar f)},
\eeqe

it remains to treat the term $\norm{\bar f - \tilde f}$. As,
$$
 \norm{\bar f - \tilde f} \leq  \betasumm \norm{\pfu - \vfu},
$$ 
and the end of the proof follows using Assumption $\mathbf{(H_{s})}$ and Theorem 1.
  \hfill $\square$

\bibliographystyle{apalike}
\bibliography{biblio_boosting}

\vskip .65cm
\noindent
Institut de Math\'ematiques de Toulouse, 118, route de Narbonne
F-31062 Toulouse Cedex 9, France
\vskip 2pt
\noindent
magali.champion@math.univ-toulouse.fr
\vskip 2pt
\noindent
Institut de Math\'ematiques de Toulouse, 118, route de Narbonne
F-31062 Toulouse Cedex 9, France
\vskip 2pt
\noindent
gaelle.chastaing@math.univ-toulouse.fr
\vskip 2pt
\noindent
Institut de Math\'ematiques de Toulouse, 118, route de Narbonne
F-31062 Toulouse Cedex 9, France
\vskip 2pt
\noindent
sebastien.gadat@math.univ-toulouse.fr
\vskip 2pt
\noindent
Universit\'e Joseph Fourier, LJK/MOISE BP 53, 38041 Grenoble Cedex, France
\vskip 2pt
\noindent
clementine.prieur@imag.fr
\vskip .3cm
%\centerline{(Received xxx 200?; accepted xxx 200?)}\par
\end{document}